\documentclass[11pt,reqno]{amsart}
\usepackage{mathrsfs}
\usepackage{hyperref}
\makeatletter
\usepackage{amsfonts}
\usepackage{amsmath}
\usepackage{amssymb}
\usepackage{multicol}
\usepackage{stmaryrd}
\usepackage{cite}
\usepackage{epsfig}
\usepackage{color}
\usepackage{graphics}
\usepackage{graphicx}
\usepackage{epstopdf}

\allowdisplaybreaks

\newcommand{\R}{\mathbb{R}}

\newcommand{\N}{\mathbb{N}}

\newcommand{\cuad}{{\sqcap\kern-.68em\sqcup}}

\newcommand{\s}{s}
\newcommand{\loc}{{\rm loc}}

\newcommand{\cO}{\mathcal{O}}
\newcommand{\cQ}{\mathcal{Q}}

\newcommand{\bH}{\mathbb{H}}
\newcommand{\bG}{\mathbb{G}}
\newcommand{\bO}{{\bf O}}

\newcommand{\bI}{{\bf I}}
\newcommand{\bu}{{\bf u}}

\numberwithin{equation}{section}

\newtheorem{theorem}{Theorem}[section]
\newtheorem{lemma}[theorem]{Lemma}

\newtheorem{definition}[theorem]{Definition}
\newtheorem{remark}[theorem]{Remark}
\newtheorem{proposition}[theorem]{Proposition}
\begin{document}

\title[The fractional Fisher-KPP equation]{The fractional  Fisher-KPP  equation with free boundaries}
 \author[H. Chen, Y. Du, W. Ni]{Huyuan Chen$^{1,2}$,\quad  Yihong Du$^{3}$\quad  and\quad  Wenjie Ni$^{3}$}
\thanks{\hspace{-.5cm}
\mbox{\   } $^1$ Center for Mathematics and Interdisciplinary Sciences, Fudan University, Shanghai 200433, PR China.\\
\mbox{\ }  $^2$ Shanghai Institute for Mathematics and Interdisciplinary Sciences, Shanghai 200433, PR China\\
\mbox{\ \ \ } Email: chenhuyuan@simis.cn. }
\thanks{\hspace{-.6cm}
\mbox{\   } $^3$ School of Science and Technology, University of New England, Armidale, NSW 2351, Australia.\\
\mbox{\ \ \ }  Emails: ydu@une.edu.au (Du), wni2@une.edu.au (Ni).}

\date{\today}

\maketitle

%%%%%%%%%%%%%%%%%%%%%%%%%%%%%%%%%%%%%%%%%%%%%%%%%%%%%%%%%%%%%%%%%%%%%%%%%%%%%%%%%%%%%%

\begin{abstract}
In this paper, we consider a free boundary problem with fractional diffusion $(-\Delta)^s$ $(0<s<1)$, as a model for species spreading, which can be viewed as a natural extension of the free boundary model  in \cite{CDLL2019}, where the nonlocal diffusion term is defined via a continuous integrable kernel function $J(x)$. This fractional diffusion model can also be viewed as a nonlocal version of the free boundary model  in \cite{DuLin2010}, whose local diffusion term is given by the classical Laplacian.

We first establish the global existence and uniqueness of the solution, which is the main contribution of this paper. This well-posedness question has been open for some time now due to the  difficulties caused by the singularity and non-integrability of the associated kernel function of $(-\Delta)^s$, and the lack of general enough regularity results for operators extending the standard fractional Laplacian.  A crucial step that enables us to answer this question relies on  an approximation approach, which breaks an associated linear fractional parabolic initial boundary value problem with curved boundaries into a sequence of  approximating problems, where the functions representing the curved boundaries are replaced by approximating step functions. Making use of existing interior regularity estimates of \cite{FR2017} and boundary estimates arising from the heat kernel estimates of \cite{CKS}, we are able to show that the solutions of the approximating problems converge to the desired solution in the limit. This method has independent interest, and should have applications elsewhere.

We then show that, for Fisher-KPP type nonlinear growth terms,  the long-time dynamics of the model exhibits a spreading-vanishing dichotomy, similar to the earlier models of \cite{DuLin2010, CDLL2019}.  More precise descriptions of the spreading profile of the solution will be considered in a separate work.

\bigskip


\medskip


\end{abstract}

\setcounter{equation}{0}

\section{Introduction}

Fractional diffusion has attracted considerable attention in many scientific fields due to its ability to model nonlocal interactions and anomalous diffusion phenomena. Unlike classical diffusion models, which are based on the standard Laplacian $\Delta$ and assume that individuals in the diffusion process disperse according to the Brownian motion law, fractional diffusion employs the fractional Laplacian operator \( (-\Delta)^s \)\ ($0<s<1$), which is capable of providing a more accurate representation of processes where   long-range dispersal occurs. It has found widespread applications in diverse fields, including physics \cite{Metzler2000,Klafter2005}, finance \cite{Scalas2006,Cartea2007}, and, notably, biological and ecological systems \cite{BRR2011,Valdinoci2017}, where traditional models are  inadequate to capture some  nonlocal effects of the dispersal processes occurring in the real-world.

For fixed $s \in (0,1)$ and $x\in\R^N$,  the  fractional Laplacian operator
$(-\Delta)^s $ is   defined by (among several equivalent forms)
$$(- \Delta)^s v(x):=c_{N,s} \lim_{\epsilon\to 0^+} \int_{\R^N\setminus B_\epsilon(x)} \frac{v(x)-v(y)}{|x-y|^{N+2s}}  dy, $$
where $B_\epsilon(x):=\{y\in \R^N: |y-x|<\epsilon\}$, and
\begin{align*}\label{norm const1}
	c_{N,\s} :=4^{\s}\pi^{-\frac N2}\frac{\Gamma(\frac{N+2\s}2)}{|\Gamma(-\s)|}.
\end{align*}

We note that the associated kernel function $J_{N,s}(x):=\frac{C_{N,s}}{|x|^{N+2s}}$ has a singularity at $x=0$ and
\[\displaystyle \int_{\R^N} J_{N,s}(x)dx=\infty.
\]

The main purpose of this research is to establish the well-posedness and then investigate the long-time behaviour of the following free boundary problem with fractional Laplacian in one space dimension:
\begin{equation}\label{eq 1.1}
	\begin{cases}
		\partial_t u +d(-\Delta)^s  u = f(t,x,u), & t>0,\ x\in \bigl(g(t),h(t)\bigr),\\[3pt]
		u(t,x)=0, & t>0,\ x \in \mathbb{R}\setminus \bigl(g(t),h(t)\bigr),\\[3pt]
		\displaystyle h'(t)= \mu \int_{g(t)}^{h(t)}\int_{h(t)}^{+\infty}\frac{u(t,x)}{|x-y|^{1+2s}}dy\,dx, & t>0,\\[3pt]
		\displaystyle g'(t)=-\mu \int_{g(t)}^{h(t)}\int_{-\infty}^{g(t)}\frac{u(t,x)}{|x-y|^{1+2s}}dy\,dx, & t>0,\\[3pt]
		h(0)=-g(0)=h_0, \ \ u(0,x)=u_0(x), & x\in \R,
	\end{cases}
\end{equation}
where $s\in(0,1)$, $\mu,d,h_0>0$, and  $f$ satisfies some natural conditions to be specified later.

Problem \eqref{eq 1.1} can be viewed as a model for the propagation of a species with density $u(t,x)$ and population range  $[g(t), h(t)]$ in a one-dimensional space. It
is a natural extension of the following nonlocal free boundary problem:
\begin{equation}
\left\{
\begin{aligned}
&\partial_t u-d\int_{\R}J(x-y) [u(t,y)-u(t,x)]\mathrm{d} y=f(u), \ 
& &t>0,~x\in(g(t),h(t)),\\
&u(t,x)=0,& &t>0,\  x\in\R\setminus (g(t),h(t)),\\
&h'(t)=\mu \int_{g(t)}^{h(t)}\int_{h(t)}^{+\infty} u(t, x)   J(x-y) \mathrm{d} y\mathrm{d} x, & &t>0, \\
&g'(t)=-\mu \int_{g(t)}^{h(t)}\int_{-\infty}^{g(t)} u(t, x)   J(x-y) \mathrm{d} y\mathrm{d} x,& &t>0,\\
&h(0)=-g(0)=h_0,\ u(0,x)=u_0(x), & &x\in \R,
\end{aligned}
\right.
\label{nonlocal}
\end{equation}
which was first considered
in \cite{CDLL2019} (see also  \cite{CQW} for $f(u)\equiv 0$), and further developed in \cite{dlz2021, DN-ma, DN-jems} among others.
In \eqref{nonlocal}, the
kernel function $J(x)$ is assumed to satisfy
\begin{description}
\item[(J)] $J\in C(\mathbb R)\cap L^\infty(\mathbb R),\; J\geq 0,\; J(0)>0,~\displaystyle \int_{\mathbb{R}}J(x)
    {\mathrm{d} x}=1$, $J$  is even.
\end{description}
Roughly speaking, $J(x-y)$ can be regarded as
 the probability that an individual at location $y$ moves to location $x$.
 
Under the assumption that the growth function $f(u)$ satisfies
\begin{description}
\item[(f)] $\begin{cases} \mbox{$f\in C^1$,\; $f>0=f(0) = f(1)$ in $(0,1)$, $f'(0)>0>f'(1)$,}\smallskip\\
\mbox{$ f(u)/u$ is decreasing for $u>0$,}
\end{cases}$
\end{description}
problem \eqref{nonlocal} has been rather well understood.  It was shown in \cite{CDLL2019} that \eqref{nonlocal} has a unique  solution triple $(u(t,x), g(t), h(t))$ defined for all time $t>0$, and as time goes to infinity,  the species modelled by
	\eqref{nonlocal} either spreads successfully, or vanishes; namely, as $t\to+\infty$, either
	\begin{itemize}
		\item {\bf Spreading:}
		$\begin{cases} (g(t), h(t))\to (-\infty, +\infty),\\
			u(t,x)\to 1 \mbox{ locally uniformly in } x\in\mathbb R,
		\end{cases}$\\
		or
		
		\item {\bf Vanishing:} $\begin{cases} (g(t), h(t))\to (g_\infty, h_\infty), \mbox{ which is a finite interval},\\
			u(t,x)\to 0 \mbox{  uniformly for } x\in [g(t), h(t)].
		\end{cases}$
	\end{itemize}
	More precise results on the spreading behaviour of \eqref{nonlocal} was obtained in \cite{dlz2021, DN-ma, DN-jems} and other works.

In \eqref{nonlocal}, the equations governing the evolution of the population range $[g(t), h(t)]$ are:
\[
 \begin{cases} \displaystyle u(t, h(t))=0,\ \  h'(t)=\mu \int_{g(t)}^{h(t)}\int_{h(t)}^{+\infty} u(t, x)   J(x-y) \mathrm{d} y\mathrm{d} x,\medskip\\
\displaystyle u(t,g(t))=0,\ \ g'(t)=-\mu \int_{g(t)}^{h(t)}\int_{-\infty}^{g(t)} u(t, x)   J(x-y) \mathrm{d} y\mathrm{d} x.
\end{cases}\]
 The quantity
\[
J_h(t):=\int_{g(t)}^{h(t)}\int_{h(t)}^{+\infty} u(t, x)   J(x-y) \mathrm{d} y\mathrm{d} x
\]
represents  the population flux across the right boundary of $[g(t), h(t)]$ at time $t$ (see \cite{CDLL2019}), which is lost in the diffusion process, and $h'(t)=\mu J_h(t)$ means that the rate of increase of the population range at its right front is proportional to this flux. $J_g(t)$ is similarly defined and explained.

 If  $J(x) \;=\; \frac{c_{1,s}}{|x|^{1+2s}}$ for $
x \;\in\; \mathbb{R}\setminus \{0\}$, then we could rewrite  $(- \Delta)^s v(x)$  as $\displaystyle \int_{\R} J(x-y)[v(x)-v(y)] dy$, and  \eqref{eq 1.1} then formally looks the same to \eqref{nonlocal}. As in \cite{CDLL2019}, for a spreading species with population density $u(t,x)$  and range $[g(t), h(t)]$,  the total population mass leaving the interval \(\bigl[g(t), h(t)\bigr]\) through its right boundary \(x = h(t)\) at time \(t\), caused by the dispersal governed by $\displaystyle\int_{\R} J(x-y)[u(t,x)-u(t,y)] dy$,
 is given by
\[
\int_{g(t)}^{h(t)} \int_{h(t)}^{\infty}
J(x - y)\,u(t,x)\,\mathrm{d}y\,\mathrm{d}x=\frac 1{2s}\int_{g(t)}^{h(t)} u(t,x) [h(t)-x]^{-2s}\,\mathrm{d}x.
\]
Since we assume \(u(t,x) = 0\) for \(x \notin [g(t),h(t)]\), the mass crossing \(x = h(t)\) is  lost to the dispersal process. We may call this, as in \cite{CDLL2019}, the outward flux at \(x = h(t)\), denoted by \(J_h(t)\). Similarly, the outward flux at \(x = g(t)\) is
\[
J_g(t) \;:=\;
\int_{g(t)}^{h(t)}
\int_{-\infty}^{g(t)}
J(x - y)\,u(t,x)\,\mathrm{d}y\,\mathrm{d}x= \frac 1{2s}\int_{g(t)}^{h(t)} u(t,x) [x-g(x)]^{-2s}\,\mathrm{d}x.
\]
Then the free boundary conditions in \eqref{eq 1.1} can be interpreted as assuming that
the  advance rate of the front is proportional to the outward flux:
\[
g'(t) = -\mu\,J_g(t),\ \ \
h'(t) = \mu\,J_h(t),
\]
 as in \eqref{nonlocal}.

However, the theory for \eqref{nonlocal} does not apply to \eqref{eq 1.1} since the kernel function $J(x) = \frac{c_{1,s}}{|x|^{1+2s}}$  has a singularity at $x=0$ and is not integrable over $\R$, doubly violating the assumption ${\bf (J)}$. To extend the theory for \eqref{nonlocal} to \eqref{eq 1.1} is a very natural question but it has remained  open  for many years, due to significant technical difficulties caused by these special features of $\frac{c_{1,s}}{|x|^{1+2s}}$.

Both \eqref{nonlocal} and \eqref{eq 1.1} can be viewed as extensions of the local diffusion model
\begin{equation}\label{A}
		\left\{\begin{array}{ll}
			\partial_t u-d u_{x x}=f(u), & t>0,\, g(t)<x<h(t), \\
			u(t, g(t))=u(t, h(t))=0, & t>0, \\
			g'(t)=-\mu u_x(t, g(t)),\ h^{\prime}(t)=-\mu u_{x}(t, h(t)), & t>0, \\
			g(0)=-h_0, \ h(0)=h_{0}, \, u(0, x)=u_{0}(x), & -h_0 \leq x \leq h_{0},
		\end{array}\right.
	\end{equation}
	which was first considered
in \cite{DuLin2010} for a special $f(u)$, and then extended to more general $f(u)$ in \cite{DL}.  Here, the evolution of the free boundaries $x=g(t)$ and $x=h(t)$ is governed by the equations in the second and third lines of \eqref{A},
	which coincide with the Stefan conditions used in classical one-phase free boundary problems describing the melting of ice in contact with water (\!\!\cite{Rub}). In a biological context, where very few first principles are available to help with the modelling, one may deduce these equations from some reasonable biological assumptions (see, for example, \cite{BDK}), whose principle is along the same line of thinking to \eqref{nonlocal} and \eqref{eq 1.1}: $h'(t)$ and $g'(t)$ are proportional to the population loss at the  boundary of the population range $[g(t), h(t)]$.
	
	It follows from  \cite{DuLin2010} and \cite{DL} that, for $f$ satisfying ${\bf (f)}$,  \eqref{A} has a unique solution, and a spreading-vanishing dichotomy holds for its long-time dynamics. 	More accurate results on the spreading profile of \eqref{A} was obtained in \cite{DMZ}.
The well-posedness of \eqref{A} was established by	a rather standard approach, where a straightening of the free boundary technique plays a central role.
Such a technique turns out to be difficult to extend to \eqref{eq 1.1}, at least for now, since the regularity theory for operators extending the standard fractional Laplacian is not yet general enough as for those extending the classical Laplacian (which was essential for the straightening  of the free boundary technique to work for \eqref{A}).
\smallskip

In this paper, we develop a  new approach to treat  \eqref{eq 1.1},  which allows us to prove that \eqref{eq 1.1} has a unique solution defined for all $t>0$. Moreover, for Fisher-KPP type growth function $f$, we show that the spreading-vanishing dichotomy exhibited by both \eqref{nonlocal} and \eqref{A} also holds for \eqref{eq 1.1}. More precise results for the long-time dynamics of \eqref{eq 1.1} will be proved in a separate  work.
		\medskip
		
		To put this work into perspective against a broader background, we briefly recall some related developments in the theory of propagation dynamics. This will be done in the next subsection, after which, the main results of this paper will be  described with necessary details.

\subsection{The Fisher-KPP model and propagation dynamics}
Let us start by recalling the classical Fisher-KPP equation, which is
a simple looking fundamental model for propagation dynamics,  first used by Fisher \cite{Fisher} and Kolmogorov--Petrovskii--Piskunov (KPP) \cite{KPP} to model the spreading of a species over $\R$, and has the following form
\begin{equation}\label{kpp}
	\begin{cases}
		\displaystyle \partial_t u -d   u_{xx} = f(u), &  t>0,\ x \in \R,\\[3pt]
		u(0,x) = u_0(x), & x \in \mathbb{R}.
	\end{cases}
\end{equation}
Here $u(t,x)$ stands for the density of the spreading species at time $t$ and spatial location $x$, and the initial population $u_0(x)$ is assumed to be a nonnegative continuous function with compact support.
	
	It was shown  in \cite{Fisher, KPP} that \eqref{kpp} admits traveling wave solutions of the form \( u(t,x) = \Phi(x - ct) \) if and only if \( c \geq c_0 \), where \( c_0 =2\sqrt{d\,f'(0)}\) is called the minimal speed.
	More importantly, \( c_0 \) determines the asymptotic spreading speed of solutions with compactly supported initial data; to be precise, for any $0<\epsilon\ll 1$, the solution $u(t,x)$ of \eqref{kpp} satisfies (see \cite{AW78})
	\[\begin{cases}
		\lim_{t\to+\infty} \sup_{|x|\geq (c_0+\epsilon)t} u(t,x)=0,\\
		\lim_{t\to+\infty}\sup_{|x|\leq (c_0-\epsilon)t}|u(t,x)-1|=0.
	\end{cases}
	\]
	These conclusions on $u(t,x)$ imply that for any $\sigma\in (0,1)$, the set
	\[
	\Omega_\sigma(t):=\{x\in\mathbb R: u(t,x)>\sigma\}
	\]
	spreads to the entire space $\mathbb R$ with asymptotic speed $c_0$. 
	The set $\Omega_\sigma(t)$ can be interpreted  biologically  as the population range of the species at time $t$
	modelled by \eqref{kpp}. The exhibition of a propagation speed is a fundamental feature of \eqref{kpp}, representing a phenomenon observed in numerous real world examples of species spreading by using the record of their range expansion  (\!\!\cite{Skellam, SK}). Such a  behaviour also arises in other areas, such as spreading of flames in combustion theory \cite{Ka}, etc.
	
	These pioneering works set the foundation for much of the multi-faceted later developments for propagation dynamics based on parabolic equations,  which form a considerable part of current research of modern PDE theory and applications, including, but not limited to, research on a broad range of problems with various nonlinearities and inhomogeneous media; see, as a small sample, \cite{BH03, BHM,dgm,FZ,FM,HNRR16,LZ,P2,shen10,W82,W02,Xin} and the references therein. This paper can be viewed as a continuation of these efforts along the line of free boundary models.

Apart from the extensions of \eqref{kpp} mentioned in the previous paragraph, there are also extensive research focusing on extensions of the local diffusion operator. 
	To capture nonlocal dispersal factors in real-world diffusion processes, some nonlocal versions of  \eqref{kpp}  have been  studied in recent years, where the local diffusion term $du_{xx}$ is replaced by a suitable nonlocal diffusion operator. One such operator involves a continuous  kernel function satisfying ${\bf (J)}$, and
under such a  change, \eqref{kpp} becomes
\begin{equation}\label{Cauchy}
\begin{cases}
\displaystyle \partial_t u-d\int_{\mathbb R}J(x-y)[u(t,y)-u(t,x)]{\mathrm{d}}y=f(u),\ 
& t>0,~x\in\mathbb R,\\
\displaystyle u(0,x)=u_0(x), &  x\in\mathbb R.
\end{cases}
\end{equation}

Problem \eqref{Cauchy} and its many variations have been extensively studied in the literature; see, for example, \cite{AC,BCV,BHR,FT, Garnier, LZ,  Roq, Yagisita2009} and the references therein. In particular, if {\bf (J)} and {\bf (f)} are satisfied, and if the nonnegative initial function $u_0$ has non-empty compact support, then the basic long-time dynamical behaviour of \eqref{Cauchy} is given by
\[
\lim_{t\to+\infty} u(t,x)=1 \ \ \mbox{ locally uniformly for $x\in\mathbb R$}.
\]
Similar to \eqref{kpp}, one can use $\tilde\Omega_\rho(t):=\{x: u(t,x)>\rho\}$ $(0<\rho<1)$ to represent  the population range at time $t$ modelled by \eqref{Cauchy}, and use its evolution in $t$ to understand the propagation dynamics  of \eqref{Cauchy}. Let
\[
    E_\rho(t):=\partial\tilde\Omega_\rho(t)=\{x\in\mathbb R: u(t,x)=\rho\}, \   x_\rho^+(t):=\sup E_\rho(t) \ \mbox{ and } \; x_\rho^-(t):=\inf E_\rho(t).
\]
It turns out that as $t\to+\infty$, $|x^{\pm}_\rho(t)|$ may go to $\infty$ linearly in $t$ or super-linearly in $t$, depending on whether the following threshold condition is satisfied by the kernel function, apart from {\bf (J)},
\begin{description}
\item[$(\bf J_{thin})$] \  There exists $\sigma>0$ such that
$
\displaystyle\int_{\R} J(x)e^{\sigma x}{\mathrm{d}}x<\infty,
$
\end{description}
which is often called a ``thin tail" condition for $J$.

More precisely,
suppose  that {\bf (f)} and {\bf (J)} are satisfied, then
\[\label{c*}
    \lim_{t\to+\infty}\frac{x_\rho^{\pm}(t)}{t}=\begin{cases}\pm c_* & \mbox{ if $(\bf J_{thin})$ holds},\\
    \pm\infty&  \mbox{ if $(\bf J_{thin})$ does not hold},
    \end{cases}
\]
where $c_*>0$ is uniquely determined by the associated traveling wave problem; see  \cite{W82, Yagisita2009} for details.
\smallskip

The fractional Laplacian $(-\Delta)^s$ is another widely used  nonlocal diffusion operator for propagation dynamics.
For example,  \cite{BRR2011} explored its role in species dispersal within fragmented habitats,  \cite{Valdinoci2017} demonstrated the advantages of nonlocal dispersal in competitive ecological settings, showing that species adopting such strategies may gain a survival advantage in resource-scarce or highly heterogeneous environments. Similarly,  \cite{Pellacci2018} analyzed optimal dispersal strategies in spatially heterogeneous landscapes via the principal eigenvalue of  fractional Neumann problems, demonstrating that species may favour local or nonlocal dispersal depending on habitat fragmentation.  \cite{Meleard2014} performed an asymptotic analysis of population models incorporating a fractional Laplacian with local or nonlocal reaction terms. %showing that under specific rescalings, the stable state invades the unstable state exponentially, with dynamics sometimes converging to a Hamilton-Jacobi equation.
Additional studies, such as \cite{Valdinoci2024,Roquejoffre2012,Annizzotto2021}, further explored the impact of fractional dispersal on ecological dynamics and mathematical modelling approaches.

 When the classical diffusion term is replaced by the fractional Laplacian in \eqref{kpp}, we obtain
\begin{equation}\label{1}
	\begin{cases}
		\displaystyle \partial_t u + d(-\Delta)^s u = f(u), \ &  t>0,\ x \in \bigl(g(t),h(t)\bigr),\\[3pt]
		u(0,x) = u_0(x), & x \in \mathbb{R}.
	\end{cases}
\end{equation}

For  \eqref{1},  \cite{Cabre2013} showed that the front propagation, measured by the position of the boundary of the level set
$\tilde\Omega_\rho(t):=\{x: u(t,x)>\rho\}$ $(0<\rho<1)$,  grows exponentially in time.   \cite{Cabre2012} examined similar equations in periodic media, discovering that the exponential propagation speed is influenced by a periodic principal eigenvalue and remains direction-independent. Additionally, \cite{Souganidis2019} explored front propagation in a generalized KPP-type fractional problem in periodic media, obtaining both the propagation speed and the asymptotic profile of solutions.

\vspace{0.5em}
\noindent
%{\bf Motivation for free boundary problems.}

However, the models \eqref{kpp}, \eqref{Cauchy} and \eqref{1} do not
 provide a clear delineation of the population range. They all use
 \[
	\Omega_\sigma(t):=\{x\in\mathbb R: u(t,x)>\sigma\}
	\]
	to approximate the population range, which
	spreads to the entire space $\mathbb R$ as time goes to infinity, with finite asymptotic speed  (for \eqref{kpp} and \eqref{Cauchy} with a thin-tailed kernel) or with infinite asymptotic speed (for \eqref{Cauchy} whose kernel is not thin-tailed and for \eqref{1}).
	 	
	From a biological point of view, a more natural choice of the population range at time $t$ is $\Omega_0(t):=\{x\in\mathbb R: u(t,x)>0\}$, but  due to the strong maximum principle for \eqref{kpp}, \eqref{Cauchy} and \eqref{1}, $\Omega_0(t)\equiv \mathbb R$ for $t>0$ even though $\Omega_0(0)$ is bounded by assumption. Therefore to capture the propagation dynamics via these models, it is necessary to use $\Omega_\sigma(t)$ rather than $\Omega_0(t)$ to stand for the population range.
	However, it is easily seen that  for any fixed $t>0$, $\Omega_\sigma(t)\to\mathbb R$ as $\sigma\to 0$. Therefore these models do not give the precise population range for positive time, although it captures the crucial propagation rate, which is independent of $\sigma$.
	
	The corresponding free boundary models \eqref{A}, \eqref{nonlocal} and \eqref{eq 1.1} fill this gap by explicitly giving the population range $[g(t), h(t)]$, while retaining the capacity of capturing  the asymptotic spreading profile. For \eqref{A}, the precise spreading behaviour was obtained in \cite{DMZ}, while  precise spreading behaviour for \eqref{nonlocal} was obtained in  \cite{dlz2021, DN-jems, DN-ma}. The purpose of this research is to establish a similar theory for \eqref{eq 1.1}.

\subsection{Main results}
We now describe the main results for \eqref{eq 1.1} obtained in this paper with all the needed details.

We start by specifying the assumptions.
For the initial function \(u_0\), we require \(u_0 \in \mathcal I_0(\cO_0)\), where $\cO_0:=(-h_0, h_0)$ and
\begin{equation}\label{102}
	\mathcal I_0 (\cO_0): =\Big\{w\in  C(\R): w\gneqq 0  \mbox{ in } \cO_0,\ w \equiv 0\ \ {\rm in}\ \, \R\setminus \cO_0  \Big\}.
\end{equation}
The growth term $f: \mathbb{R}^+\times\mathbb{R}\times\mathbb{R}^+
\rightarrow\mathbb{R}$ is assumed to  satisfy

\begin{description}
	\item[(f1)]  For fixed $\tau\geq 0$, $f(\cdot,\cdot, \tau)\in C^{\frac{\alpha}{2s},\alpha}(\R^+\times\R)$ for some $\alpha\in (0,1)$, and\\
	\mbox{\ \ \ \ \ \  }$f(t,x,0)\equiv 0$,\  $f(t,x,\tau)$
	is locally Lipschitz in $\tau\in\R^+$ uniformly in $(t,x)$
	\footnote{Namely, for any $K>0$, there exists a constant $L=L(K)>0$ such that \\ \mbox{\ \ \ \ \ }	$\left|f(t,x,\tau_1)-f(t,x,\tau_2)\right|\le L\, |\tau_1-\tau_2|$ for $\tau_1, \tau_2\in [0, K]$ and $(t,x)\in \R^+\times \R$.};	
\end{description}
\begin{description}
	\item[(f2)] There exists
	$K_0>0$ such that $f(t,x,\tau)\leq 0$ for $\tau\ge K_0$ and $(t,x)\in \R^+\times \R$.
\end{description}

\begin{theorem}[Well-posedness]\label{teo 1}
	Assume that  $f$ verifies  {\bf  (f1)} and {\bf (f2)};
then for any $u_0\in \mathcal I_0 (\cO_0)$,  problem
	\eqref{eq 1.1} has a unique solution $(u, g,h)$ defined for all $t>0$.
\end{theorem}

To analyze the long-term dynamical behavior of \eqref{eq 1.1}, for simplicity and clarity, we impose additional restrictions on $f$. To be precise,  the following additional Fisher-KPP type conditions will be assumed.

\begin{description}
	\item[(f3)] $f=f(u)$ is independent of $(t,x)$, $f'(0)$ exists and $f'(0)>0$.\smallskip
	\item[(f4)]  $f(u)/u$ is non-increasing  and $f(u)/u<f'(0)$ for $u\in (0,\infty)$.
	 \end{description}

\begin{theorem}[Spreading--vanishing dichotomy]\label{teo 2.2}
	Assume that   $f$ satisfies conditions {\bf (f1)} through to {\bf (f4)}, and  $u_0\in \mathcal I_0(\cO_0)$.
	Let $(u,g,h)$ be the unique solution of problem \eqref{eq 1.1}. Then the following dichotomy holds:

\hspace{0.3cm} {\rm either (i)} {\bf (vanishing)}\ \ \   $\begin{cases}\displaystyle \lim_{t\to\infty} (g(t), h(t))=(g_\infty, h_\infty) \mbox{ is a finite interval,}\\
	\displaystyle \lim_{t\to\infty}u(t,x)= 0 \mbox{ uniformly in $x\in\R$, } \end{cases}$
	
	\hspace{0.3cm} {\rm or (ii)} {\bf (spreading)}\ \ \   $\begin{cases} \displaystyle\lim_{t\to\infty} (g(t), h(t))=(-\infty, \infty),\\
	\displaystyle \lim_{t\to\infty}u(t,x)= u_* \mbox{ locally uniformly in $x\in\R$}, \end{cases}$ \\
	where $u_*$ is the unique positive zero point of $f(u)$.
\end{theorem}

\begin{theorem}[Spreading--vanishing criteria]\label{teo 2}
Assume that the hypotheses of Theorem \ref{teo 2.2} hold, and let $\lambda_{s,1}(\cO_0)$ denote the first eigenvalue of $(-\Delta)^s$ over $\cO_0$ subject to the zero Dirichlet boundary conditions. Then the following conclusions hold:
	\begin{itemize}
		\item[{\rm (i)}] If $\lambda_{s,1}(\cO_0)\leq f'(0)/d$, then spreading occurs regardless of the choice of $\mu>0$ and $u_0\in\mathcal I_0(\cO_0)$.
		\item[{\rm (ii)}] If $\lambda_{s,1}(\cO_0)> f'(0)/d$, then there exists a threshold value $\mu^*\in (0,+\infty)$, depending on $u_0$, such that spreading occurs if $\mu>\mu^*$, and vanishing occurs if $\mu\in (0,\mu^*]$.
	\end{itemize}
\end{theorem}

When spreading happens, more precise descriptions of the spreading profile of the solution are considered in a separate work.

\vspace{0.5em}

The most difficult part of this work lies in the proof of Theorem \ref{teo 1}, which relies on a  new approach, strikingly different from those for the related free boundary models \eqref{nonlocal} and \eqref{A}. The key  is an approximation process, which  appears in Section 2, where a related linear fractional parabolic problem over a region with  curved boundaries is treated. This approximation approach is based on two existing results for the fractional Laplacian, one is the interior estimate in {\cite{FR2017}, the other is a boundary estimate with the constant in the bound independent of the boundary (a consequence of the heat kernel estimate of \cite{CKS}), which is a key fact that enables the idea to work for the associated semilinear problem treated in Section 3 and for the free boundary problem treated in Section 4. The uniqueness proof in Section 4 is also very different from the existing ones for \eqref{A} and \eqref{nonlocal}, where the Banach contraction mapping theorem is the core; here the uniqueness proof is based on an entropy function (see Section 4.2).

The proof of Theorems \ref{teo 2.2} and \ref{teo 2} are given in Section 5, where existing ideas for \eqref{A} and \eqref{nonlocal} are combined with techniques developed specifically for \eqref{eq 1.1} here to prove the desired conclusions.

 \setcounter{equation}{0}
\section{Linear  problem with given curved boundaries}

%In this section we show  the heat kernel under the Dirichlet boundary condition
For convenience, we first introduce some notations. For given $h_0\in(0,\infty),\,  T\in(0,\infty)$ we define
\begin{align*}
&\mathbb H_{h_0, T}:=\Big\{h\in C([0,T])\!:~h(0)=h_0,
\; h(t) \mbox{ is nondecreasing in } [0,T] \Big\},\\
&\mathbb G_{h_0, T}:=\Big\{g\in C([0,T])\!:-g\in\mathbb{H}_{h_0, T}\Big\},\\
& \cO_t:=(g(t),h(t)) \mbox{ for } t\in [0, T],\\
& C_0(\bar \cO_0):=\Big\{u\in C(\R)\!: \ u(x)=0 \mbox{ for } x\in\R\setminus \cO_0\Big\},\\
&\mathcal I_0(\cO_0):=\Big\{u\in C_0(\bar \cO_0)\!: u\gneqq  0 \mbox{ in } \cO_0\Big\}.
\end{align*}
  For $t>0$ and $x\in \cO_t$, we define
 $$ \rho(t,x):=d(x, \partial \cO_t)=\min\big\{h(t)-x,\, x-g(t)\big\}.$$
Moreover, for $g\in \mathbb G_{h_0, T}$, $h\in\mathbb H_{h_0, T}$ and  nonnegative $u_0\in C_0(\bar\cO_0)$, we define
\begin{align*}
&\Omega_{g, h}=\Omega_{g, h,T}:=\left\{(t,x)\in\mathbb{R}^2: 0<t\leq T,~g(t)<x<h(t)\right\},\\
&\mathbb{X}=\mathbb X_{u_0,g,h}:=\Big\{\phi\in C(\overline\Omega_{g,h})~:~\phi\ge0~\text{in}
~\Omega_{g,h},~\phi(0,x)=u_0(x)~\text{for}~x\in [-h_0,h_0]~\\
& \hspace{6cm} \text{and} \;\; \phi(t,g(t))=
\phi(t,h(t))=0~\ \text{for }\ 0\le t\le T\Big\}.
\end{align*}

For given $g\in \mathbb G_{h_0, T}$, $h\in\mathbb H_{h_0, T}$ with $0<T<\infty$,  $  \eta \in C^{\frac{\alpha}{2s},\alpha}(  \Omega_{g, h}) \cap L^{\infty}(  \Omega_{g, h})$, $0<\alpha<1$, and $w_0\in C_0(\bar\cO_0)$,
consider the Cauchy problem
\begin{equation}\label{eq 2.1-h}
\left\{
\begin{array}{lll}
\partial_t w +d(-\Delta)^s  w =\eta  \quad &{\rm for}\ \,   (t,x)\in \Omega_{g,h},\\[3mm]
\qquad\quad\   w(t,x)=0  \quad  &{\rm for}\ \, (t,x)\in \big((0,T]\times \R\big) \setminus \Omega_{g,h},\\[3mm]
\qquad\quad\  w(0,\cdot)=w_0  &{\rm in}\ \    \cO_0 .
\end{array}\right.
 \end{equation}

  \begin{definition}[Classical solution and weak solution]\label{weak sol} Let \( \alpha, s \in (0,1) \) be such that \( \frac{\alpha}{2s} \in (0,1) \) and \( \alpha + 2s \) is not an integer.
\begin{itemize}
	\item[{\rm (i)}]  A function $u$ is called a classical solution of  \eqref{eq 2.1-h}
	if   $u\in C\big([0,T]\times \mathbb R\big)\cap  C_{\loc}^{1+\frac{\alpha}{2s} ,2s+\alpha}(\Omega_{g,h}) $
	satisfies  \eqref{eq 2.1-h} point-wisely. Here  $u\in C_{\rm loc}^{1+\frac{\alpha}{2s} ,2s+\alpha}(\Omega_{g,h})$ means that  for any $Q\Subset \Omega_{g, h}\ ($i.e., $Q$ is a compact subset of $\Omega_{g, h})$,   $$   \|u\|_{C^{1+\frac{\alpha}{2s}}_t(	Q)} + \|u\|_{C^{2s+\alpha}_x(Q)} <\infty. $$

	\item[{\rm (ii)}]  A function $u$ is called a weak solution of  \eqref{eq 2.1-h}
	if the following  hold:\smallskip
	
	\noindent {\rm (a)}  $u \in L^{1}\big( (0, T]; \, L^\infty(\mathbb R)\big) \cap  C\big([0, T];\, L^2(\R)  \big)$  with $u(t,x)=0$ for $x\in \mathbb{R}\backslash \cO_t$ and $t\in(0,T]$ \medskip
	
	\noindent {\rm (b)}  for any $\varphi \in  C^{1, 2s}\big(\overline \Omega_{g,h}\big)$ satisfying $\varphi (t,\cdot)\in C_c^{2s}(\cO_t)$,  and  any $T'\in(0, T]$,
{\small	\begin{equation}\label{def w1}
		\int_{0}^{T^{\prime}}  \int_{\cO_t} \left[- \partial_t \varphi  +   (-\Delta)^s
		\varphi -\eta \right] u \, d x\,  d t  =\int_{\R} \varphi(0, x) w_0(x)\,d x-\int_{\R} u\left(T^{\prime}, x \right)
		\varphi\left(T^{\prime},x \right) d x.
	\end{equation}
	}
	
\end{itemize}
  \end{definition}

 %For $m\geq0$, we denote
%  $$L^\infty\big((0,T]; H^{ ms}_0(\cO_t)\big)= \Big\{u\in L^2\big([0,t]\times\R \big)\!:\,  \sup_{\tau\in(0,T]} \|u(\tau,\cdot)\|_{sm, \cO_\tau}<+\infty\Big\}. $$

 \begin{proposition}[Existence and estimates of solutions to  \eqref{eq 2.1-h}]\label{KG lm 1}  Let \( \alpha \in (0,1) \) and $s\in(0,1)$ be such that \( \frac{\alpha}{2s} \in (0,1) \) and \( \alpha + 2s \) is not an integer.
Assume that $T \in(0,+\infty) $,     $(g,h)\in\bG_{h_0, T}\times \bH_{h_0, T}$,    $w_0\in  \mathcal I_0 (\cO_0)$,   $  \eta \in C^{\frac{\alpha}{2s},\alpha}_{\rm loc}(  \Omega_{g, h}) \cap L^{\infty}(  \Omega_{g, h})$ and $\eta\geq 0$.
Then
 problem \eqref{eq 2.1-h}
 admits a unique classical  solution $u\in C\big([0, T]\times \R  \big)\cap  C_{\loc }^{1+\frac{\alpha}{2s} ,2s+\alpha}(\Omega_{g,h})$, and it satisfies
   \begin{align}\label{ab-2}
	 {\bf (boundary\ estimates)} \quad 0\leq u(t,x) \leq C_0(1\land \frac{ \rho^s (t,x)}{\sqrt{t}}) \quad \text{for } t\in (0,T],\ x \in (g(t),h(t)),
   \end{align}
   where the constant  $C_0>0$ depends only on $\|\eta\|_\infty,\ \|w_0\|_\infty$,\ $s$ and $T$.

 \end{proposition}

 To prove Proposition \ref{KG lm 1}, some preparations are needed.
 We start with a maximum principle. For  $T\in(0,\infty)$, let $ \bO_t$ with $t\in[0,T]$   be a bounded interval for $x$ such that
 $$\bO_{t}\subset \bO_{t'} \quad {\rm if}\ \, 0\leq t\leq t'\leq T$$
 and
$$
 \cQ_T=\big\{(t,x)\in\R^2\!:\, t\in (0,T],\,   x\in \bO_t  \big\}.
$$
 \begin{lemma}[Maximum principle]\label{lm comp-1}
 Let $s\in(0,1)$,  $T\in(0,+\infty) $,   $u \in
C(\overline\cQ_T )$ and for every $t\in (0, T]$, $u(t,\cdot)\in  L^1(\R, \frac{dx}{(1+|x|)^{1+2s}})$.
Suppose $\partial_tu, (-\Delta)^s u \in  C\big( \cQ_T \big)$ and 
\begin{equation}\label{eq 2.2-stronger+}
\left\{ \arraycolsep=1pt
\begin{array}{lll}
\displaystyle \partial_t u+  (-\Delta)^s  u+{C}(t,x)u\geq 0\quad \  &{\rm in}\ \      \cQ_T,\\[2mm]
 \phantom{ (-\Delta)^s -- \  \,  }
\displaystyle   u  \geq 0 \quad \ &{{\rm in}}\  \  \big((0,T]\times  \R \big)  \setminus   \cQ_T, \\[2mm]
 \phantom{ (-\Delta)^s   \   \, }
 u(0,\cdot) \geq 0 \quad \ &{{\rm in}}\  \ \bO_0
\end{array}
\right.
\end{equation}
in the point-wise sense, where  $\|{C} \|_{\infty}<\infty$.
Then $ u\geq 0 \ {\rm in}\ \, \cQ_T.$
Furthermore,
$u>0\ {\rm in}\    \cQ_T $
if 	
\begin{equation}\label{side}
u(t,\cdot) \gneqq0 \mbox{ in  }  \R\setminus \bO_t  \mbox{ for every } t\in (0, T],
\end{equation}
or
 \begin{equation}\label{t=0}
 \mbox{$u(0, \cdot)\gneqq0$ in $\bO_0$.}
 \end{equation}
  \end{lemma}

\noindent{\bf Proof.}  Fix $M\geq \|C\|_\infty$ and let $\tilde u:=ue^{-M t}$.  Then $\tilde u$  satisfies
\begin{align*}
	\displaystyle \partial_t  \tilde u+  (-\Delta)^s  \tilde u+[C(t,x)+M]\tilde u=e^{-Mt}[\partial_t u+  (-\Delta)^s  u+{C}(t,x)u ]\geq 0\quad \  {\rm in}\ \      \cQ_T.
\end{align*}
Hence we may assume that  $C\geq 0$.

To prove the conclusion, we employ a contradiction argument. Assume that there exists $(t_0,x_0)\in \cQ_{T} $  such that
$$u(t_0,x_0)=\min_{(t,x)\in  \overline\cQ_{T}}u(t,x)< 0,$$
then  $u(t , x)\geq u(t_0,x_0)$ for $(t,x)\in[0,T]\times \R$ and $t_0>0$.
By direct computation, we have  $\partial_tu(t_0, x_{0})\leq 0$  and
 \begin{align*}
 (-\Delta)^s u(t_0, x_{0})&=c_{1,s} \int_{\R}\frac{u(t_0, x_{0})-u(t_0,y)}{|x_0-y|^{1+2s}}  \, d y     \leq c_{1,s} \int_{\R\setminus    \cO_{t_0}} \frac{ u(t_0, x_{0}) }{|x_0-y|^{1+2s} }   d y < 0,
 \end{align*}
hence
$$\partial_tu(t_0, x_{0})+  (-\Delta)^s  u(t_0, x_{0})+{C} u(t_0,x_0)<0,$$
which contradicts  the assumption.
Thus $u$ is nonnegative in $\overline \cQ_{T}$.

We next prove that $u>0$ in $\cQ_T$ if \eqref{side} or \eqref{t=0} holds. Suppose that  \eqref{side} holds and we proceed by an indirect argument.
So we assume that $u(t_0,x_0)=0$ for some $x_0\in \bO_{t_0}$ and $t_0\in (0, T]$.
Then   $\partial_t u(t_0, x_{0})\leq 0$ and
$$
 (-\Delta)^s   u(t_0, x_{0})= -c_{1,s} \int_{\R} \frac{  u(t_0, x_0+y)}{|y|^{1+2s} }  dy < 0,
$$
since  $u(t_0,\cdot) \gneqq   0$ in $\R\setminus \bO_{t_0}$ by \eqref{side}. It follows that
$$0\leq \partial_tu(t_0, x_{0})+  (-\Delta)^s  u(t_0, x_{0})+{C} u(t_0,x_0)<0,$$
which is the desired contradiction.

It remains to consider the case \eqref{t=0}.
In this case, we consider the solution $\bar u$ of the following auxiliary problem
 \begin{equation}\label{eq 2.2-stronger++}
\left\{ \arraycolsep=1pt
\begin{array}{ll}
\displaystyle \partial_t u+  (-\Delta)^s  u+{C}u=0\quad \  &{\rm in}\ \   (0,T]\times \bO_0,\\[2mm]
\displaystyle   u  = 0  \ &{{\rm in}}\  \    (0,T]\times \big(\R \setminus \bO_0\big), \\[2mm]
  u(0,\cdot) =u_0  \ &{{\rm in}}\  \ \bO_0.
\end{array}
\right.
\end{equation}
The just proved part of the comparison principle implies that
$$u\geq \bar u\geq 0\quad{\rm in}\ \, (0,T]\times \bO_0. $$

We show  that $\bar u>0$ in  $(0,T]\times \bO_0$. Let $\lambda_{s,1}(\bO_0)$ be the principal eigenvalue of $(-\Delta)^s$ over $\bO_0$ with Dirichlet boundary conditions, and $\varphi_{s,1}\geq 0$ an associated eigenfunction.
Multiplying the first equation in \eqref{eq 2.2-stronger++} by $e^{Mt}\varphi_{s,1}$ and integrating over $\bO_0$, we obtain, after a simple manipulation,
 \begin{align*}
0 &=   \partial_t\int_{\bO_0}e^{Mt}\bar u (t,x) \varphi_{s,1}(x)dx+\int_{\bO_0}e^{Mt}\varphi_{s,1}(x)(-\Delta)^s\bar u (t,x) dx\\
&\ \ \ \ \ \ +\int_{\bO_0}e^{Mt}[C(t,x)-M]\varphi_{s,1}(x)\bar u (t,x) dx
\\
&  \leq \partial_t\int_{\bO_0} e^{Mt}\bar u (t,x) \varphi_{s,1}(x)dx+\lambda_{s,1}(\bO_0)\int_{\bO_0}e^{Mt}\bar u (t,x) \varphi_{s,1}(x) dx,
  \end{align*}
which implies
$$\int_{\bO_{0}} \bar u (t,x) \varphi_{s,1}(x)dx\geq e^{-t \lambda_{s,1}(\bO_0)-tM} \int_{\bO_0}u_0(x) \varphi_{s,1}(x)dx>0, $$
and so $\bar u(t,\cdot)\gneqq 0$  in $\bO_0$ for every $t\in (0, T]$.

If there exists $(t_0, x_0)\in (0,T]\times \bO_0$ such that $\bar u (t_0,x_0)=0$, then $\partial_t\bar u (t_0,x_0)\leq 0$ and so $(-\Delta)^s\bar u(t_0,x_0)=-\partial_t \bar u(t_0,x_0)\geq 0$. However, the definition of $(-\Delta)^s$ implies
 \begin{align*}
	(-\Delta)^s \bar u(t_0, x_{0})&=c_{1,s} \int_{\R}\frac{ \bar u(t_0, x_{0})- \bar u(t_0,y)}{|x_0-y|^{1+2s}}  \, d y     \leq -c_{1,s} \int_{ \bO_0} \frac{  \bar u(t_0, y) }{|x_0-y|^{1+2s} }   d y < 0,
\end{align*}
which is a contradiction. Hence, $u\geq \bar u>0$ in  $(0,T]\times \bO_0$.

To complete the proof, it remains to show that $u>0$ in
$\cQ_T$. For each $t_0\in (0, T)$, by the proved fact we have $u(t_0,x)>0$ in $\bO_0$. We may now
 apply the above conclusion with $t=0$ replaced by $t=t_0$ to obtain $u(t,x)>0$ in $(t_0, T]\times \bO_{t_0}$. Since $\cup_{t_0\in (0, T)}\left[(t_0, T]\times \bO_{t_0}\right]=\cQ_T$, it follows that $u>0$ in $
\cQ_T$. The proof is now complete.
 \hfill$\Box$\medskip

 %\begin{remark}
 % Lemma \ref{lm comp-1} also holds when $T=\infty$ since in that case the conclusion holds for any $T\in (0, \infty)$.
 %\end{remark}

We also need the following regularity  result of \cite{FR2017}:

 \begin{lemma}\label{lemmajfa} {\rm (\!\cite[Theorems 1.1 and 1.3]{FR2017})}
	Suppose that \( s \in (0,1) \)  and \( u \) is a weak solution to
	\begin{equation*}
		\partial_t u + d(-\Delta)^s u = f \quad \text{in } (0,\sigma_1) \times B_{\sigma_2},
	\end{equation*}
	where \( B_{\sigma_2} = (-\sigma_2, \sigma_2) \). Let \( \alpha \in (0,1) \) be such that \( \frac{\alpha}{2s} \in (0,1) \) and \( \alpha + 2s \) be not an integer, and
	\begin{equation*}
		C_\alpha := \|u\|_{L^{\infty}((0,\sigma_1) \times B_{\sigma_2})} + \|f\|_{C^{\frac{\alpha}{2s},\alpha}((0,\sigma_1) \times B_{\sigma_2})}<+\infty.
	\end{equation*}
	Then, for any $\epsilon\in (0, \sigma_1)$,
	\begin{equation*}
		\|u\|_{C_t^{1+\frac{\alpha}{2s}}((\epsilon,\sigma_1) \times B_{\sigma_2})} + \|u\|_{C_x^{2s+\alpha}((\epsilon,\sigma_1) \times B_{\sigma_2})} \leq C
	\end{equation*}
	for some constant \( C \) depending only on \( \epsilon, \sigma_1, \sigma_2 \), \( s \), \( \alpha \)  and $C_\alpha $.
\end{lemma}

Here, following the notations of \cite{FR2017},
\[\begin{cases}
\|f\|_{C^{\frac{\alpha}{2s},\alpha}((0,\sigma_1) \times B_{\sigma_2})}=\|f\|_{C^{\frac{\alpha}{2s},\alpha}([0,\sigma_1] \times \overline B_{\sigma_2})},\\[2mm]
\|u\|_{C_t^{1+\frac{\alpha}{2s}}((\epsilon,\sigma_1) \times B_{\sigma_2})} :=\displaystyle \sup_{x\in B_{\sigma_2}}\|u(\cdot, x)\|_{C^{1+\frac{\alpha}{2s}}([\sigma,\sigma_1] \times \overline B_{\sigma_2})},\\[2mm]
\|u\|_{C_x^{2s+\alpha}((\epsilon,\sigma_1) \times B_{\sigma_2})} :=\displaystyle \sup_{t\in (\epsilon, \sigma_1)}\|u(t, \cdot)\|_{C^{2s+\alpha}([\sigma,\sigma_1] \times \overline B_{\sigma_2})}.
\end{cases}
\]

We further require the following result on the Dirichlet heat kernel of $(-\Delta)^s$ over an interval of $\R$.

 \begin{lemma}\label{lem-hk} {\rm (\!\!\cite[Theorem 1.1]{CKS})}
	Suppose that \( s \in (0,1) \) and $I=(l_1, l_2)$ is a bounded interval in $\R$. Then the Dirichlet heat kernel  of $(-\Delta)^s$ over $I$, denoted by
	${\bf H}^s_I(t,x,y)$, has the following estimates:
	\begin{itemize}
	\item[(i)] For each $T>0$, there exist constants $c_1\leq c_2$ such that
	\begin{equation}\label{hk-est1}
	c_1\leq \frac{{\bf H}_I^s(t,x,y)}{\Big(1\land  \frac{d_I(x)^s}{\sqrt t}\Big)\Big(1\land  \frac{d_I(y)^s}{\sqrt t}\Big)\Big(t^{\frac{-1}{2s}}\land \frac{t}{|x-y|^{1+2s}}\Big)}\leq  c_2 \mbox{ for } t\in (0, T],\ x,y\in I,
	\end{equation}
	where we  have used the notations $d_I(x):=\min_{z\in\partial I}|x-z|$ and $u\land v:=\min\{u, v\}$.
	\item[(ii)] For each $T>0$, there exist constants $c_1\leq c_2$ such that
	\begin{equation}\label{hk-est2}
	c_1\leq \frac{{\bf H}^s_I(t,x,y)}{e^{-\lambda^s_1 t/|I|^{2s} }d_I(x)^sd_I(y)^s}\leq c_2 \mbox{ for } t\geq T,\ x,y\in I,
	\end{equation}
	where   $\lambda^s_1$ is the first Dirichlet eigenvalue of $(-\Delta)^s$ on $[0,1]$.
	\end{itemize}
	\end{lemma}
	\begin{remark}\label{rm-hk}
	By Theorem 2.4 of \cite{CKS}, the constant
	  $c_2$ in Lemma \ref{lem-hk} part {\rm (i)} is independent  of the size of the interval $I$. This fact will be frequently used in this paper.	\end{remark}

	\begin{remark}\label{rm-hk-0}
If $\{(\lambda_{k},\phi_{k})\}$ is the complete normalised sequence  of  Dirichlet eigenvalues  and associated eigenfunctions  of $(-\Delta)^s$ over $I$,
 i.e. they are all the nontrivial solutions of
  \begin{equation}\label{eq 2.1-eig-r}
\left\{
\begin{array}{lll}
\!\! (-\Delta)^s  \phi =\lambda \phi   \ \  &{\rm in}\ \,  I,\\[2mm]
\quad\  \,   \phi=0  \ \ &{\rm in}\ \,   \R \setminus I,
\end{array}\right.
 \end{equation}
 with
 $$
\lambda_{1} <\lambda_{2}  \le \cdots\le \lambda_{k}  \le \lambda_{k+1} \le \cdots,\, \ \ \ \ \int_{I_0}\phi_i\phi_jdx=\begin{cases} 1,& i=j,\\
0, & i\not=j,\end{cases}
$$
then
	$${\bf H}_{I}^s(t,x,y)=\sum_{k=1}^{\infty}e^{-\lambda_{k}t}\phi_{k}(x)\phi_{k}(y). $$
	\end{remark}
	Using Lemma \ref{lem-hk}, we have the following result, which will play a central role in the proof of Proposition \ref{KG lm 1}.

\begin{lemma}\label{lem-I}
	Suppose that $T>0$,  \( s \in (0,1) \), $I=(l_1, l_2)$ is a bounded interval in $\R$, $w_0\in C_0(\bar I)$ and $f\in L^\infty([0,T]\times I)$. Then the solution of
	\begin{equation}\label{eq-I}
\left\{
\begin{array}{lll}
\partial_t w +d(-\Delta)^s  w =f   \quad &{\rm in}\ \,     (0, T]\times I,\\[3mm]
\qquad\quad\,   w(t,x)=0  \quad  &{\rm for}\ \,  (t,x)\in \big((0,T]\times (\R\setminus  \bar I\big) ,\\[3mm]
\qquad\quad\,   w(0,\cdot)=w_0  & {\rm in}\ \,  I
\end{array}\right.
 \end{equation}
satisfies $w\in C([0,T]\times \bar I)$ and
\[
w(t,x)\to w_0(x) \ \mbox{ as } t\to 0^+ \mbox{ uniformly in } \bar I.
\]
Moreover,  there exists some $C_0>0$ depending only on $\|w_0\|_\infty$, $ \|f\|_\infty$, $s$ and $T$  such that
\[
 |w(t,x)|\leq  C_0 (1\land \frac{ d_I(x)^s}{\sqrt{t}}) \  \mbox{ for } t\in (0, T], \ x\in I. \]
 \end{lemma}
\begin{proof}
Let ${\bf H}^s$ denote the heat kernel of $(-\Delta)^s$ over $\R$, and for convenience we always assume
\[
\mbox{${\bf H}^s_I(t,x,y)=0$ \ for $(x,y)\not\in I\times I$,\ \quad  $w_0(x)=0$ for $x\not\in I$.}
\]
 It is well known that
\[
0\leq {\bf H}^s_I(t,x,y)\leq {\bf H}^s(t,x,y) \ \mbox{ for } t>0,\ x,\ y\in\R.
\]
Moreover,
\[
\int_{\R} {\bf H}^s(t,x,y)dy=1 \ \mbox{ for all } t>0,\ x\in\R,
\]
and for any $x\in I$ and $\epsilon>0$,
\[
\lim_{t\to 0^+}\int_{I} {\bf H}^s_I(t,x,y)dy=\lim_{t\to 0^+}\int_{B_\epsilon(x)\cap I} {\bf H}^s_I(t,x,y)dy=1,
\]
where $B_\epsilon(x)=(x-\epsilon, x+\epsilon)$.

By \eqref{hk-est1} we obtain, for $x\in I$,
\begin{align*}
 0\leq \int_{I\setminus B_\epsilon(x)}  {\bf H}_I^s(t,x,y)dy&\leq c_2  \int_{I\setminus B_\epsilon(x)}  (t^{\frac{-1}{2s}}\land \frac{t}{|x-y|^{1+2s}})dy\\
&\leq c_2\int_{I\setminus B_{\epsilon}(x)}\frac{t}{|x-y|^{1+2s}}dy\\[2mm]
&\leq t\, c_2\, |I| \epsilon^{-1-2s}\to 0 \ \mbox{ as } t\to 0^+.
\end{align*}
Hence
\[
\lim_{t\to 0^+}\int_{I\setminus B_\epsilon(x)} {\bf H}^s_I(t,x,y)dy=0.
\]

We will use these properties and the  following formula to obtain the desired estimates of the solution:
\begin{equation}\label{hk-sol}
    w(t,x)
    = \int_{I}  {\bf H}_I^s(t,x,y)w_0(y)dy
    + \int_0^t\int_{I} {\bf H}_I^s(t-\tau,x,y) f(\tau,y) \,dyd\tau.
\end{equation}

For any given $\delta>0$ small  there exists $\epsilon>0$  such that $|w_0(z)-w_0(y)|\leq \delta$ for all $y\in B_\epsilon(z)$ and $z\in\R$. Then for  any $x\in I$ we have, by the above estimates,
\begin{align*}
\lim_{t\to 0^+} |w(t,x)-w_0(x)|&=\lim_{t\to 0^+}\left |\int_{I}  {\bf H}_I^s(t,x,y)[w_0(y) -w_0(x)]dy\right|
 \\
 &=\lim_{t\to 0^+}\left |\int_{B_\epsilon(x)\cap I}  {\bf H}_I^s(t,x,y)[w_0(y) -w_0(x)]dy\right|\\
&\leq \limsup_{t\to 0^+} \int_{B_\epsilon(x)\cap I}  {\bf H}_I^s(t,x,y)|w_0(y)-w_0(x)|dy\\
&\leq \delta \limsup_{t\to 0^+} \int_{B_\epsilon(x)\cap I} {\bf H}_I^s(t,x,y)=\delta.
\end{align*}
Since $\delta>0$ can be arbitrarily small, it follows that
\begin{equation}\label{t-0}
\lim_{t\to 0^+} w(t,x)=w_0(x).
\end{equation}
Since $w(t,x)=0=w_0(x)$ for $x\in \partial I$, the above limit holds trivially for  $x\in\partial  I$.

By \eqref{hk-est2} and \eqref{hk-sol} we see that $w(t,x)$ is continuous over $(0, T]\times \bar I$. This, together with \eqref{t-0} and the continuity of $w_0(x)$ implies that $w(t,x)$ is continuous over $[0,T]\times \bar I$ and hence uniformly continuous over this compact set. It follows that $w(t,x)$ is equicontinuous for $(t,x)\in [0, T]\times \bar I$. In particular,
$w(t,x)\to w_0(x)$ uniformly for $x\in\bar I$ as $t\to 0^+$.

By \eqref{hk-est1} we obtain
\begin{align*}
 \int_{I}  {\bf H}_I^s(t,x,y)dy&\leq c_2 (1\land \frac{ d_I(x)^s}{\sqrt{t}}) \int_I  (t^{\frac{-1}{2s}}\land \frac{t}{|x-y|^{1+2s}})dy\ \,
\mbox{ for } x\in I,\ t\in (0, T].
\end{align*}
Since
\begin{align*}
\int_I  (t^{\frac{-1}{2s}}\land \frac{t}{|x-y|^{1+2s}})dy&\leq \int_{I\cap B_{t^{\frac 1{2s}}}(x)}  t^{\frac{-1}{2s}}dy+\int_{I\setminus B_{t^{\frac 1{2s}}}(x)}\frac{t}{|x-y|^{1+2s}})dy\\
&\leq 2+2t\int_{t^{\frac{-1}{2s}}}^\infty r^{-1-2s}dr=2+\frac 1 s,
\end{align*}
we have
\begin{align*}
 \int_{I}  {\bf H}_I^s(t,x,y)dy&\leq c_2 (2+\frac 1 s)(1\land \frac{ d_I(x)^s}{\sqrt{t}})\ \, \mbox{ for } x\in I,\ t\in (0, T],
\end{align*}
and hence
\[
 \left|\int_{I}  {\bf H}_I^s(t,x,y)w_0(y)dy\right|\leq C_1(1\land \frac{ d_I(x)^s}{\sqrt{t}})\ \, \mbox{ for } x\in I,\ t\in (0, T],
 \]
 with $C_1>0$ depending only on $\|w_0\|_\infty$, $s$ and $T$.
 Moreover,
\[
\int_0^t\int_{I} {\bf H}_I^s(t-\tau,x,y) \,dyd\tau\leq c_2 (2+\frac 1 s)d_I(x)^s\int_0^t \tau^{-1/2}d\tau=C_2t^{1/2} d_I(x)^s\ \, \mbox{ for } x\in I,\ t\in (0, T].
\]
 It now follows by \eqref{hk-sol} that
\[
|w(t,x)|\leq C_1 (1\land \frac{ d_I(x)^s}{\sqrt{t}})+C_2\|f\|_\infty t^{1/2} d_I(x)^s\leq C_0 (1\land \frac{ d_I(x)^s}{\sqrt{t}})
\]
for all $t\in (0, T]$, $x\in I$ and some $C_0>C_1$, depending only on  $\|w_0\|_\infty,\  \|f\|_\infty$, $s$ and $T$.
\end{proof}

We will use an approximation process to prove the existence of a solution to (\ref{eq 2.1-h}). For this purpose, for any given $T>0,\ (g,h)\in\bG_{h_0, T}\times \bH_{h_0, T}$ and  $n\in \N$ (the  set of positive integers), we define the following sequences:
\begin{align*}
&t_{n,i}:=\frac{i}{2^n} T, \qquad \cO_{{n,i}} :=(g(t_{n,i}), h(t_{n,i})),\ i=0, ..., 2^n,
\\
&\cQ_{{n,t}}:=\cO_{{n,i}}\ {\rm for }\ t\in (t_{n,i},t_{n,i+1}],\quad
\cQ_n:=\bigg\{(t,x): t\in (0,T],\ x\in \cQ_{{n,t}}\bigg\}.
\end{align*}

Note that each $\cQ_n$ is the union of $2^n$ rectangular regions in the $(x,t)$-plane, stacked on top of each other, with the one above no smaller in size than the one below, due to the  monotonicity of $g$ and $h$. Moreover,
$$\cQ_n\subset \cQ_{n+1}\subset   \Omega_{g,h}\quad \text{ for every $n\in\N$ \ \ and}\quad \Omega_{g,h}= \bigcup^{\infty}_{n=1} \cQ_n.$$
Therefore, $\{\cQ_n\}$ is an approximation sequence of $\Omega_{g,h}$.

Next we show that for any fixed $n\geq 1$,  the following approximation problem of \eqref{eq 2.1-h},
\begin{equation}\label{eq 2.1-h-n}
\left\{
\begin{array}{lll}
\partial_t w +d(-\Delta)^s  w =\eta   \quad &{\rm in}\ \,      \cQ_n,\\[3mm]
\qquad\quad\,   w(t,x)=0  \quad  &{\rm in}\ \,  \big((0,T]\times \R\big) \setminus \cQ_n,\\[3mm]
\qquad\quad\,   w(0,\cdot)=w_0  & {\rm in}\ \,  \R
\end{array}\right.
 \end{equation}
 admits a unique classical solution $u_n$ for suitably given $\eta$ and $w_0$. Let us note that in Proposition \ref{KG lm 1}, both $\eta$ and $w_0$ are assumed nonnegative.
 In Lemma \ref{lem-I} above, as well as in Lemma \ref{lemma2.5} below, these non-negativity assumptions are not required.

\begin{lemma}\label{lemma2.5}
	Let \( \alpha \in (0,1) \) be such that \( \frac{\alpha}{2s} \in (0,1) \) and \( \alpha + 2s \) is not an integer. Assume that
	$ \eta \in C^{\frac{\alpha}{2s},\alpha}(  \Omega_{g, h})   \cap L^\infty(\Omega_{g, h})$ and $w_0\in C_0(\bar \cO_0)$. Then the following conclusions are valid:    	
	\begin{itemize}
		\item[{\rm (i)}] 	Problem \eqref{eq 2.1-h-n} admits a unique solution \( u_n \in C\big([0, T] \times \mathbb{R}  \big) \cap C^{1+\frac{\alpha}{2s},2s+\alpha}_{\loc} ( Q_n) \).		
		\item[{\rm (ii)}] 	Let
		\begin{align*}
			\Omega_{g,h}^{\delta} = \{(t,x) : t \in (\delta,T],\ x \in (g(t) + \delta, h(t) - \delta)\}.
		\end{align*}
		Then, for $\delta>0$ small enough and all sufficiently large \( n \),  there exists a constant \( C=C(\delta, \eta) \)  independent of \( n \),  such that
		\begin{align}
			  \|u_n\|_{C^{1+\frac{\alpha}{2s}}_t(	\Omega_{g,h}^{\delta})} + \|u_n\|_{C^{2s+\alpha}_x(	\Omega_{g,h}^{\delta})} \leq C. \label{2.16a}
			\end{align}
	Moreover, there exists a constant $C_0>0$ depending on $\|\eta\|_\infty$, $\|w_0\|_\infty$, $s$ and $T$ but independent of $n$, such that
	\begin{align}		
			 |u_n(t,x)| \leq C_0(1\land\frac{ \rho_n^s (t,x)}{\sqrt{t}}) \quad \text{for } x \in \cQ_{n,t},\label{2.17a}
		\end{align}
	 where \[ \rho_n(t,x) :=d(x,\partial \cQ_{n,t})= {\rm min}\big(|x-g(t_{n,i})|, |x-h(t_{n,i})|\big) \] with $i$ determined by  $t\in (t_{n,i},t_{n,i+1}];$
	 in particular, there holds	
		\begin{align}\label{un-rho}
	 	 |u_n(t,x)| \leq C_0(1\land\frac{  \rho^s (t,x)}{\sqrt{t}})\quad \text{in }  (0,T]\times \R,
		\end{align}
		 where 		 \[
		  \rho(t,x)=\begin{cases} %d(x,\partial (g(t), h(t)))=
		  \min\big\{h(t)-x,\, x-g(t)\big\} & \mbox{ if } x\in (g(t), h(t)),\\[2mm]
		  0 & \mbox{ if } x\not\in (g(t), h(t)).
		  \end{cases}
		  \]

	\end{itemize}
\end{lemma}

\begin{proof}	(i)
  Denote by  ${\bf H}_i$ the Dirichlet heat kernel of $(-\Delta)^s$ on $ (g(t_{n,i}), h(t_{n,i}))$, namely
  \[
  {\bf H}_i={\bf H}^s_{(g(t_{n,i}), h(t_{n,i}))}.
  \]
 For $ (t,x)\in (0,t_{n,1}]\times \R$,  define
\begin{equation}\label{eq 1.1-int-s}
    w_1(t,x)
    := \int_{\cO_{{n,0}}}  {\bf H}_1(t,x,y)w_0(y)dy
    +
    \int_0^t\int_{\cO_{{n,0}}} {\bf H}_1(t-\tau,x,y) \eta(s,y)  \,dyd\tau.
\end{equation}
Iteratively,   for $ (t,x)\in (t_{n,i-1},t_{n,i}]\times \R$ with $i=2,\cdots, 2^n$, define
\begin{equation}\label{eq 1.1-int-s}\begin{aligned}
    w_i(t,x)
    := &\int_{\cO_{{n,i-1}}}  {\bf H}_i(t-t_{n, i-1},x,y)w_{i-1}(t_{n,i-1},y)dy\\
    & \ +
    \int_{t_{n, i-1}}^t\int_{\cO_{{n,i-1}}} {\bf H}_i(t-\tau ,x,y) \eta(\tau,y)   \,dyd\tau.
\end{aligned}\end{equation}
 Now  set
 $$ u_n(t,x):=w_i(t,x)\quad {\rm if} \ \, t\in (t_{n,i-1},t_{n,i}], \ \ x\in\R, \ i=1,..., 2^n. $$
 By Lemma \ref{lem-I} we easily see that $u_n$ is a weak solution of \eqref{eq 2.1-h-n}. Moreover, using Lemma \ref{lemmajfa}, we further see that
 \eqref{eq 2.1-h-n} is satisfied point-wisely in $\cQ_n$, and  so $u_n$ is a classical solution.
 The uniqueness of $u_n$ now follows from the maximum principle (see Lemma \ref{lm comp-1}).

 \smallskip
	
	(ii)
	Consider the function
	\[\mbox{\( v(t,x) := e^{\eta_* t} K_1 \), where \( K_1 =\|w_0\|_\infty+1\) and \( \eta_* = \|\eta\|_\infty \).}
	\]
	 This function satisfies
	\[
	\partial_t v + d(-\Delta)^s v = \eta_* v \geq \eta(t,x) \quad \text{in } (t,x) \in \cQ_n.
	\]
	By inductively applying the maximum principle Lemma \ref{lm comp-1} over the domains \( \{(t,x) : t \in (t_{i-1}, t_i],\, x \in \cQ_{n,t}\} \), we deduce that \( u_n \leq e^{\eta_* t} K_1 \) for \( (t,x) \in \bar{\cQ}_n \). Similarly we can show \( u_n \geq -e^{\eta_* t} K_1 \) for \( (t,x) \in \bar{\cQ}_n \).
	Hence,
	\[
	\mbox{\( |u_n(t,x)|\leq C_1 = e^{\eta_* T} K_1 \) for \( (t,x) \in \bar{\cQ}_n \).}
	\]
	
	 Since \( 	\Omega_{g,h}^{\delta} \) and $\cQ_n$ converge to \( \Omega_{g,h} \) as $\delta\to 0$ and $n\to \infty$, respectively,  we have for sufficiently small \( \delta \) and large $n$,
	\[
	\Omega_{g,h}^{\delta} \subset \Omega_{g,h}^{\delta/2} \subset \cQ_n.
	\]
	Clearly, for small \( \epsilon > 0 \) depending on \( \delta \),
	%\( \inf_{t \in [0,T]} |g'(t)| \), and \( \inf_{t \in [0,T]} |h'(t)| \),
	there exist finitely many points \( \{(t_i, x_i)\}_{i=1}^N \subset \Omega_{g,h}^{\delta/2} \) such that
	\[
	\Omega_{g,h}^{\delta} \subset \bigcup_{i=1}^N (t_i + \epsilon, t_i + 2\epsilon) \times B_{\epsilon}(x_i)  \subset \bigcup_{i=1}^N (t_i , t_i + 2\epsilon) \times B_{2\epsilon}(x_i) \subset \Omega_{g,h}^{\delta/2}.
	\]
	Given that \( |u_n| \leq C_1 \) and for some $C_2>0$
	$$ \|\eta\|_{C^{\frac{\alpha}{2s},\alpha}(\Omega_{g,h}^{\delta/2} )}\leq C_2, $$
 applying Lemma \ref{lemmajfa} to each cylinder \( (t_i, t_i + 2\epsilon) \times B_{2\epsilon}(x_i) \), we obtain the interior estimates of \( u_n \) as stated in \eqref{2.16a},
 with $C$ depending on $\delta$ but not on $n$.
	
	It remains to prove \eqref{2.17a}.
	For any \( t' \in (0,T] \),  consider the following problem over the rectangle domain $[0,t']\times \cQ_{n, t'}$:
	\[
	\begin{cases}
		\partial_t w + d(-\Delta)^s w = |\eta| & \text{for } (t,x) \in (0,t'] \times \cQ_{n,t'}, \\[3mm]
		w(t,x) = 0 & \text{for } (t,x) \in (0,t'] \times (\mathbb{R} \setminus \cQ_{n,t'}), \\[3mm]
		w(0,\cdot) = |w_0| & \text{in } \cQ_{n, t'}.
	\end{cases}
	\]
	By the maximum principle we know that $w\geq 0$. By Lemma \ref{lem-I}, there exists
 some constant $C_0>0$ depending on $\|\eta\|_\infty$, $\|w_0\|_\infty$, $s$ and $T$,  but independent of $t'$ and $n$, such that
	\[
	0 \leq  w(t,x) \leq C_0(1\land\frac{ d^s(x, \partial{\cQ_{n,t'}})}{\sqrt{t}}) \mbox{ for } \ (t,x) \in (0,t'] \times \cQ_{n,t'}.
	\]
	 By the maximum principle  Lemma  \ref{lm comp-1} we easily deduce that
	\[
	|u_n(t,x)| \leq w(t,x) \quad \text{for } (t,x) \in (0,t'] \times \cQ_{n,t'}.
	\]
Hence
	\[
	| u_n(t',x)| \leq C_0(1\land\frac{ d^s(x, \partial{\cQ_{n,t'}}}{\sqrt{t'}}) \quad \text{for } x \in \cQ_{n,t'},
	\]
	which gives \eqref{2.17a}. The proof is finished.
%This provides {\color{red}the boundary regularity \( u_{n}(t,\cdot) \in C^{s}_0(\cQ_{n,t}) \)}. The proof
\end{proof}

\begin{proof}[\bf Proof of Proposition \ref{KG lm 1}.] Since $\cQ_n\subset \cQ_{n+1}$ and $\eta\geq 0$, $w_0\geq 0$, each $u_{n+1}$ is nonnegative and is a super solution of
 (\ref{eq 2.1-h-n}) over $\cQ_n$. So the maximum principle Lemma \ref{lm comp-1} gives
 $$u_n\leq u_{n+1}\quad {\rm in}\ \, [0,T]\times \R$$
 indicating that $u_n$ is nondecreasing with respect to  $n$.
By Lemma \ref{lemma2.5}, the  sequence $u_n$ is uniformly bounded from above by a constant.
Hence, there is a unique function $u_{g,h}\in L^\infty$ such that
$$u_{g,h}(t,x)=\lim_{n\to\infty} u_n(t,x)\quad {\rm for}\ \,(t,x)\in [0,T]\times \R,$$
and
$$u_{g,h}(t,\cdot)=\lim_{n\to\infty} u_n(t,\cdot)\quad {\rm in}\ \, L^1(\R).$$

In view of  \eqref{2.16a}, and the compact  embeddings
$C^{1+\frac{\alpha}{2s}}_t(	\Omega_{g,h}^{\delta})\hookrightarrow C^{1+\frac{\alpha'}{2s}}_t(	\Omega_{g,h}^{\delta})$
and $C^{2s+\alpha}_x(	\Omega_{g,h}^{\delta})\hookrightarrow C^{2s+\alpha'}_x(	\Omega_{g,h}^{\delta})$ for any fixed $\alpha'\in (0,\alpha)$,
 we conclude that
$u_n\to u_{g,h}\quad {\rm in}\ C^{1+\frac{\alpha'}{2s},\,2s+\alpha'}(	\Omega_{g,h}^{\delta})$ and
\begin{align}\label{2.18}
\|u_{g,h}\|_{C^{1+\frac{\alpha'}{2s}}_t(	\Omega_{g,h}^{\delta})} + \|u_{g,h}\|_{C^{2s+\alpha'}_x(	\Omega_{g,h}^{\delta})} \leq C_3
\end{align}
for some $C_3=C_3(\delta)>0$.
 Therefore, we have that
$$\partial_tu_{g,h} + d(-\Delta)^s u_{g,h} = \eta  \quad  \text{for } (t,x) \in  \Omega_{g,h}. $$

By \eqref{un-rho}, we obtain
\begin{align}\label{2.19}
\quad |u_{g,h}(t,x)| \leq C_0(1\land\frac{ \rho^s (t,x)}{\sqrt{t}}) \quad \text{for } t\in (0, T],\ x \in (g(t),h(t)).
\end{align}

 Recall that $\cO_T=(g(T), h(T))$. Let $\bar w$ be the solution of
 \[
\begin{cases}
	\partial_t u + d(-\Delta)^s u =\eta & \text{for } (t,x) \in  (0,T]\times \cO_T, \\[3mm]
	u(t,x) =  0 & \text{for } (t,x) \in (0,T]\times (\R \setminus \cO_T), \\[3mm]
	u(0,\cdot) = w_0 & \text{in }\,  \cO_T.
\end{cases}
\]
Then  the maximum principle yields
 $$u_n \leq  \bar w \quad {\rm in}\ \,  (0,T]\times \R\ \,  \mbox{ for every } n\geq 1. $$
 It follows that $u_{g,h}\leq \bar w$ in $(0,T]\times \R$.
 On the other hand, for fixed $n\geq 1$, we have
 $$u_{g,h} \geq  u_n \quad {\rm in}\ \,  (0,T]\times \R. $$
 Since $\bar w$ solves the equation over a rectangle, we can use Lemma \ref{lem-I}  to see that $\bar w (t,x)\to w_0(x)$ uniformly in $x\in\R$ as $t\to 0$.
 For the same reason, we see that
  $u_n (t,x)\to w_0(x)$ uniformly in $x\in\R$ as $t\to 0$.
These imply
$$\lim_{t\to0^+} u_{g,h}(t,x)=w_0(x)\quad {\rm uniformly\  for}\ \, x\in\R. $$
Hence  $u_{g,h}$ satisfies \eqref{eq 2.1-h} in the classical sense. The uniqueness conclusion follows from the maximum principle Lemma \ref{lm comp-1}.
\end{proof}

  \setcounter{equation}{0}
 \section{Semilinear  problem with given curved boundaries}

 In this section, we consider  the following semilinear problem with curved boundaries:
  \begin{equation}\label{eq-u}
\left\{
\begin{array}{lll}
\partial_t  u +d(-\Delta)^s  u =f(t,x,u)    \quad &{\rm for}\ \,   (t,x)\in \Omega_{g,h,T},\\[3mm]
  u(t,x)=0  \quad  &{\rm for}\ \, (t,x)\in \big((0,T]\times \R\big) \setminus \Omega_{g,h,T},\\[3mm]
u(0,\cdot)=u_0  &{\rm in}\ \    \cO_0,
\end{array}\right.
 \end{equation}
 where $ \Omega_{g,h,T}$ is a fixed domain in $[0,+\infty) \times \R$, defined in the previous section.

 When $f$ satisfies ${\bf (f1)}$, we can always shrink the value of $\alpha\in (0,1)$ slightly when needed such that
 \[
  \frac{\alpha}{2s} \in (0,1) \mbox{ and \( \alpha + 2s \) is not an integer.}
 \]
 We will make this convention from now on.

 \begin{theorem}\label{thm-semi}
 	 Assume that  $f$ verifies {\bf(f1)}, {\bf(f2)},
	  $(g,h)\in\bG_{h_0, T}\times \bH_{h_0, T}$ and  $u_0\in  \mathcal I_0 (\cO_0)$.
 	Then
 	problem \eqref{eq-u}
 	admits a classical  nonnegative solution $u\in C\big([0, T]\times \R  \big)\cap  C_{\loc}^{1+\frac{\alpha'}{2s} ,2s+\alpha'}(\Omega_{g,h})$, and
		\begin{equation}\label{ab-2}
 		 u(t,x) \leq C_0(1\land \frac{\rho^s (t,x)}{\sqrt{t}}) \quad \text{for } t\in (0,T],\ x \in (g(t),h(t)),
 	\end{equation}
 	where   \(  \rho(t,x):=\min\big\{h(t)-x,\, x+g(t)\big\} \), $\alpha'\in (0,\alpha)$,  and $C_0$ is a positive constant depending only on $\|w_0\|_\infty$, $f$, $s$ and $T$.
\end{theorem}

\noindent{\bf Proof. }
 We claim that any classical  solution of (\ref{eq-u}) is uniformly bounded, namely,
 \begin{equation}\label{bound-1}
  0\leq u(t, x) \leq \max\big\{\|u_0\|_{L^\infty(\cO_0)}+1, K_0+1\big\}=:K_1 \ \ \mbox{ for } t\in (0, T],\ x\in\R,
   \end{equation}
  where $K_0$ satisfies  $f(t,x,K_0)<0$ for all $t\geq 0$ and $x\in \mathbb{R}$.

     In fact,  for any $K\geq K_1$, the constant functions $\bar u\equiv K$ and $\underline u\equiv 0$ form a pair of upper and lower solutions of (\ref{eq-u}). It then follows from the maximum principle that \eqref{bound-1} holds.

By ${\bf (f1)}$, there exists $L_0>0$ depending on $K_1$ such that
\[
|f(t,x,\tau_1)-f(t,x,\tau_2)|\leq L_0|\tau_1-\tau_2| \ \mbox{ for } (t,x)\in[0,T]\times \R,\ \tau_1,\tau_2\in [0, e^{L_0T}K_1].
\]
 Define
 \begin{equation}\label{eq lam-0}
 \ F_0(t,x,\tau):=L_0 \tau+   e^{L_0 t}f(t,x, e^{-L_0t} \tau).
 \end{equation}
 Then  $\tau\mapsto F_0(t,x,\tau)$ is nonnegative and increasing for any $(t,x)\in[0,T]\times \R$, $\tau\in [0, e^{L_0T}K_1]$, since
 \begin{align*}
  F_0(t,x,\tau_2)-F_0(t,x,\tau_1) &= L_0(\tau_2-\tau_1)+ e^{L_0 t}f(t,x, e^{-L_0 t} \tau_2)-f(t,x, e^{-L_0 t} \tau_1)\\
  &\geq L_0(\tau_2-\tau_1)-L_0|\tau_2-\tau_1|\geq 0
  {\rm\ \ for}\ \, 0\leq \tau_1\leq\tau_2\leq K_1,\ (t,x)\in[0,T]\times \R.
  \end{align*}

If $(g,h,u)$ is a solution of (\ref{eq 1.1})
and
$
\bar u(t,x)=e^{L_0 t} u(t,x)$. Then $(g,h,\bar u)$ is a solution of
 \begin{equation}\label{eq 2.1}
\left\{
\begin{array}{rll}
\partial_t  \bar u  +d(-\Delta)^s  \bar u=F_0(t,x,\bar u)  \quad &{\rm for}\ \,   (t,x)\in \Omega_{g,h},\\[3mm]
 \ \bar u(t,x)=0  \quad  &{\rm for}\ \, t\in(0,T],\ \, x\in \R \setminus\big(g(t),h(t)\big),\\[4mm]
 \   \bar u(0,x)=u_0(x) \quad &{\rm for}\ \,  x\in\cO_0.
 \end{array}\right.
 \end{equation}

{\bf Step 1:} We show the existence of a weak solution.

 We do this by an iteration procedure.
 By Proposition \ref{KG lm 1}, the     problem
 \begin{equation}\label{eq 3.1-fix-0}
 \left\{
\begin{array}{rll}
\partial_t  u  +d(-\Delta)^s  u=0  &{\rm for}\ \,   t\in \Omega_{g,h,T},\\[3mm]
 \, u(t,x)=0    &{\rm for}\ \, t\in(0,T],\ \, x\in \R \setminus (g(t),h(t)),\\[3mm]
 u(0,x)=u_0(x)  &{\rm for}\ \,  x\in[-h_0,h_0]
\end{array}\right.
 \end{equation}
 has a unique positive solution 
 $\bar u_{0} \in C\big([0, T]\times \R  \big)\cap  C_{\rm loc}^{1+\frac{\alpha}{2s} ,2s+\alpha}(\Omega_{g,h,T})$.

 It is clear that the function $U_0(t,x)=e^{L_0 t} K_1$ satisfies $U_0(0,x)=K_1\geq u_0(x)$  and
$$ \partial_t  U_0  +d(-\Delta)^s U_0=L_0 U_0\geq F_0(t,x, U_0)\geq 0\quad{\rm in}\ (0,T]\times \R.   $$
Then the maximum principle Lemma \ref{lm comp-1}  indicates that
$$\bar u_{0}\leq  U_0\leq e^{L_0T}K_1\quad{\rm in}\ \, [0,T]\times\R,$$
and hence $F(\cdot,\cdot, \bar u_0)\geq 0$.

 Repeatedly  using  Proposition \ref{KG lm 1} and the maximum principle, together with the fact that $\tau\to F_0(t,x,\tau)$ is nonnegative and increasing for $\tau\in [0, r^{L_0T}K_1]$,
 we could define $\bar u_{i}\in C\big([0, T]\times \R  \big)\cap  C_{\rm loc}^{1+\frac{\alpha}{2s} ,2s+\alpha}(\Omega_{g,h,T})$ as the unique classical solution of
 \begin{equation}\label{eq 3.1-fix-i}
 \left\{
\begin{array}{rll}
\partial_t  u  +d(-\Delta)^s  u= F_0(t,x, \bar u_{i-1})        &{\rm for}\ \,   (t,x)\in\Omega_{g,h,T},\\[3mm]
 u(t,x)=0    &{\rm for}\ \, t\in(0,T],\ \, x\in \R \setminus\big(g (t),h (t)\big),\\[3mm]
 u(0,x)=u_0(x)  &{\rm for}\ \,  x\in[-h_0,h_0],
\end{array}\right.
 \end{equation}
 and obtain
   \begin{align*}
U_0\geq  u_{i+1}\geq u_{i}\geq0\ \mbox{ for } i=0,1,2,...
 \end{align*}
 Hence $\{u_{i}\}_{i=0}^{\infty}$ is an increasing sequence of nonnegative functions, and 
  $\bar u:=\lim_{i\to+\infty}\bar u_{i}$
 is well-defined in $[0,T]\times\R$.

  Moreover,  for any $\varphi \in  C^{1, 2s}\big(\overline \Omega_{g,h, T}\big)$ satisfying $\varphi (t,\cdot)\in C_c^{2s}(\cO_t)$,
  \begin{equation}\label{weak identity-2-i}
\begin{aligned}
 &\int_{0}^{T} \int_{\R^N}\bar u_{\lambda_0,i}\Big(- \partial_t \varphi  +   (-\Delta)^s
\varphi \Big)- \varphi F_0(t,x, u_{\lambda_0, i-1})      d x\,  d t \\
&=\int_{\R^N} \varphi(0, x) w_0(x)\,d x-\int_{\R^N} u_{i}\left(T, x \right)
\varphi\left(T,x \right) d x.
\end{aligned}
\end{equation}
Passing to the limit $i\to\infty$ in (\ref{weak identity-2-i}), we see that $\bar u$  is a weak solution of  (\ref{eq 2.1}).
 Moreover,  $\bar u$ is nonnegative  and bounded.

 \smallskip

 {\bf Step 2:}  We show that $\bar u$ is a classical solution.

 For any $(T_1,T_2)\times B_r(x_0)\subset \Omega_{g,h,T}$, it follows by \cite[Theorem 1.3]{FR2017}  that
for $T_1<T_1'<T_2'<T_2$ and $\epsilon\in(0,s)$,
there exists $C>0$ such that
 \begin{align*}
&\quad\  \| \bar u_{i}\|_{C_t^{1-\frac{\epsilon}{2s}}((T_1',T_2')\times B_{\frac r2}(x_0)) } + \| \bar u_{i}\|_{C_x^{2s-\epsilon}((T_1',T_2')\times B_{\frac r2}(x_0))}
\\[1mm]&\leq C\Big(\| \bar u_{i}\|_{L^\infty((T_1,T_2)\times B_r(x_0)) }  +\|F_0(\cdot,\cdot,\bar u_{i-1}) \|_{L^\infty((T_1,T_2)\times B_r(x_0)) } \Big)
\\[1mm]&\leq CK_1,
  \end{align*}
where $C>0$ is independent of $i\geq 1$. Thus, for every $i\geq 1$ we have
 \begin{align*}
 \| \bar u_{i}\|_{C^{1-\frac{\epsilon}{2s},2s-\epsilon}_{t,x}((T_1',T_2')\times B_{\frac r4}(x_0)) } &  \leq CK_1.
  \end{align*}
These estimates and the properties of $F_0$ allow us to use \cite[Theorem 1.1]{FR2017} to conclude that
 for $T_1'<T_1''<T_2''<T_2'$ and $\alpha\in(0,1)$,  there exists $C>0$ such that, for every $i\geq 2$,
 \begin{align*}
& \| \bar u_{i}\|_{C^{1+\frac{\alpha}{2s},2s+\alpha}_{t,x}((T_1'',T_2'')\times B_{\frac r4}(x_0)) }
\\[1mm] &\leq   \| \bar u_{i}\|_{C_t^{1+\frac{\alpha}{2s}}((T_1'',T_2'')\times B_{\frac r4}(x_0)) } + \| \bar u_{i}\|_{C_x^{2s+\alpha}((T_1'',T_2'')\times B_{\frac r4}(x_0))}
\\[1mm]&\leq C\Big(\| \bar u_{i}\|_{C^{\frac{\alpha}{2s},\alpha}_{t,x}((T_1',T_2')\times B_{\frac r4}(x_0)) }  +\|F_0(\cdot,\cdot,\bar u_{i-1})\|_{C^{\frac{\alpha}{2s},\alpha}_{t,x}((T_1',T_2')\times B_{\frac r4}(x_0)) } \Big)
\\[1mm]&\leq C\Big[\| \bar u_{i}\|_{C^{\frac{\alpha}{2s},\alpha}_{t,x}((T_1',T_2')\times B_{\frac r4}(x_0)) }
   +C_{F_0}  \| \bar u_{i-1}\|_{C^{\frac{\alpha}{2s},\alpha}_{t,x}((T_1',T_2')\times B_{\frac r4}(x_0)) } \Big]
\\[1mm]&\leq C'K_1,
  \end{align*}
where $C, C'>0$ are independent of $i$.
This implies that $ \bar u \in C_{t,x}^{1+\frac{\alpha}{2s}, 2s+\alpha}(Q)$ for every $Q\Subset\Omega_{g,h,T}$.

For $\bar t\in(0,T]$,  and $\bar K:=\sup_{t\in[0,T],x\in\R,\tau\in[0,K_1]} F_0(t,x,\tau)$, let $w_{*}$  be the solution of  the problem
$$
\left\{
\begin{array}{lll}
\partial_t  u +d(-\Delta)^s  u =\bar K    \quad &{\rm for}\ \,    (t,x)\in   (0,\bar t]\times \cO_{\bar t} ,\\[3mm]
  u(t,x)=0  \quad  &{\rm for}\ \,  (t,x)\in  (0,\bar t]\times \big(\R\setminus  \cO_{\bar t} \big),\\[3mm]
u(0, x)=u_0(x)  & {\rm for}\ \,   x\in\cO_{\bar t}.
\end{array}\right.
$$
By Lemma \ref{lem-I} there exists $C_0>0$ depending only on $\|w_0\|_\infty$, $\bar K$ (which is completely determined by $f$), $s$ and $T$, such that
$$0\leq w_{*}(t,x) \leq C_0(1\land \frac{d^s(x, \partial\cO_{\bar t})}{\sqrt{\bar t}})  \quad{\rm for }\ \,  (t,x)\in(0,t']\times \cO_{\bar t}.  $$
 By the maximum principle, we have
$$\bar u_0(t,x)\leq \bar u(t,x)\leq w_{*}(t,x)   \quad{\rm for\ any}\ \,  (t,x)\in(0,\bar t]\times \cO_{\bar t}. $$
It follows that
$$\bar u(t,x)  \leq C_0(1\land \frac{d^s(x, \partial\cO_{ t})}{\sqrt{t}})  \ \mbox{ for } (t,x)\in\Omega_{g,h,T},  $$
and
\[
\bar u(t,x)\to u_0(x) \mbox{ as $t\to 0^+$  uniformly for $x\in \R$.}
\]

Therefore, we can conclude that, for some fixed $\alpha'\in (0, \alpha)$,
$$\bar u_{i}\to \bar u \quad {\rm in}\ \, C_{\loc}^{1+\frac{\alpha'}{2s}, 2s+\alpha'}(\Omega_{g,h})\cap C(\bar \Omega_{g,h}), $$
which,  combined with the fact that $\bar u_{i}=0$ in $\big((0,T]\times \R\big)\setminus \Omega_{g,h}$, allow us to pass to the limit  as $i\to+\infty$ in (\ref{eq 3.1-fix-i})  to see that
 $  \bar u$ indeed is a classical solution of \eqref{eq 2.1}. Moreover,  (\ref{ab-2}) holds.

   \smallskip

{\bf Step 3:}   We prove the uniqueness.

  To this end,  we assume that $u_1, u_2$ are classical solutions of (\ref{eq-u}).  Let
$$w(t,x)=e^{-L_0 t}\big(u_1(t,x)-u_2(t,x)\big)\quad \mbox{for } (t,x)\in(0,T]\times \R, $$
where  $L_0$ is the Lipschitz constant of $f$ introduced at the beginning of the proof.
 Then $w$ satisfies
 \begin{equation}\label{eq 2.1-sh-u}
\left\{
\begin{array}{cll}
w_t +d(-\Delta)^s  w =-L_0  w+ e^{-L_0 t_0}[f (t,x,u_1)- f (t,x,u_2) ], & (t,x)\in \Omega_{g,h,T},\\[2mm]
 w(t,x)=0,  & (t,x)\in \big((0,T]\times \R\big) \setminus \Omega_{g,h,T},\\[2mm]
 w (0,x)=0,  & x\in \cO_0.
\end{array}\right.
 \end{equation}
  Now we  show $w\equiv 0$ in $\Omega_{g,h}$.
 If  there is a point $(t_0,x_0)\in \Omega_{g,h}$ such that
 $$w(t_0,x_0)=\max_{(t,x)\in  \Omega_{g,h,T}} w(t,x)>0,$$
 then $\partial_t w(t_0,x_0)\geq0$ and
 $$(-\Delta)^sw(t_0,x_0)=c_{1,s}\int_{\R}\frac{w(t_0,x_0)-w(t,y)}{|x_0-y|^{1+2s}}>0, $$
 and hence
$$\partial_t w(t_0,x_0)+(-\Delta)^sw(t_0,x_0) >0,  $$
which, however, is a contradiction to
 \begin{align*}
 \partial_t w(t_0,x_0)+(-\Delta)^sw(t_0,x_0)&=-L_0  w(t_0,x_0)+ e^{-L_0t}\big[f_0(t_0,x_0,u_1(t_0,x_0))- f_0(t_0,x_0,u_2(t_0,x_0))\big]
\\[1mm]& \leq  -L_0  w(t_0,x_0)+L_0|w(t_0,x_0)|=0.
  \end{align*}
 Similarly $\min_{(t,x)\in  \Omega_{g,h,T}} w(t,x)<0$ also leads to a contradiction.
 As a consequence, we have  $w\equiv 0$, as desired.
\hfill$\Box$\medskip

  \setcounter{equation}{0}
\section{Free boundary problem}

 Our aim in this section is to establish the well-posedness of  (\ref{eq 1.1}) by proving Theorem \ref{teo 1}. Throughout this section, the conditions of Theorem \ref{teo 1} are always assumed.

 \subsection{Existence}
We prove the existence part of Theorem \ref{teo 1} in this subsection, which will need the following result.

\begin{lemma}\label{u-c} Suppose $(g_*,h_*), (g_n, h_n)\in \bG_{h_0, T}\times \bH_{h_0, T}$ for $n=1,2,...$.  Let $u_*$ and $u_n$ be the solution of \eqref{eq-u} with $(g, h)=(g_*, h_*)$ and $(g,h)=(g_n, h_n)$, respectively.  If $(g_n, h_n)\to (g_*, h_*)$ in $C([0, T])\times C([0, T])$ as $n\to\infty$, then
$u_n\to u_*$ in $L^\infty([0, T]\times \R)$.
  \end{lemma}
  \begin{proof} Firstly by \eqref{bound-1}, we have that
  \[
  0\leq u_n\leq K_1 \mbox{ for all } n\geq 1.
  \]
  Secondly, much as in the proof of Theorem \ref{thm-semi}, where the sequence $\{\bar u_i\}$ was shown to have a convergent subsequence by making use of the interior regularity of the solutions, we can similarly show that the sequence $\{u_n\}$ here has a convergent subsequence, which for simplicity of notation is assumed to be $\{u_n\}$ itself, such that $u_n\to \tilde u$ in $C_{\loc}^{1+\frac{\alpha}{2s}, 2s+\alpha}(\Omega_{g_*,h_*, T})$ for some $\tilde u$ and $\alpha\in (0, 1)$.

  We next use the boundary estimate for $u_n$ and $(g_n, h_n)\to (g_*, h_*)$ to show that $u_n\to\tilde u$ in $L^\infty([0, T]\times \R)$ and $\tilde u$ solves \eqref{eq-u} with $(g,h)$ replaced by $(g_*, h_*)$. This would imply $\tilde u=u_*$ by the uniqueness of the solution, which in turn implies that the entire original sequence $\{u_n\}$ converges to $u_*$ in $C_{\loc}^{1+\frac{\alpha}{2s}, 2s+\alpha}(\Omega_{g_*,h_*, T})\cap L^\infty([0, T]\times \R)$.

  We now set to prove $u_n\to\tilde u$ in $L^\infty([0, T]\times \R)$. Let $\bar u_n:=e^{L_0t}u_n$. Then $\bar u_n$ solves \eqref{eq 2.1} with $(g,h)$ replaced by $(g_n, h_n)$. Let $\bar K:=\sup_{t\in[0,T],x\in\R,\tau\in[0,K_1]} F_0(t,x,\tau)$. Since  $(g_n, h_n)\to (g_*, h_*)$ in $C([0, T])\times C([0, T])$, without loss of generality we may assume $[g_n(t), h_n(t)]\subset (g_*(t)-1, h_*(t)+1) $ for all $t\in [0, T]$ and $n\geq 1$.

  Let $w_{*}$  be the solution of
$$
\left\{
\begin{array}{lll}
\partial_t  u +d(-\Delta)^s  u =\bar K    \quad &{\rm for}\ \,    t\in (0, T], \ x\in (g_*(t)-1, h_*(t)+1),\\[3mm]
  u(t,x)=0  \quad  &{\rm for}\ \,  t\in (0, T],\ x\in \R\setminus  (g_*(t)-1, h_*(t)+1),\\[3mm]
u(0, x)=u_0(x)  & {\rm for}\ \, x\in\R,
\end{array}\right.
$$
and $v_*$ be the solution of
$$
\left\{
\begin{array}{lll}
\partial_t  u +d(-\Delta)^s  u =0    \quad &{\rm for}\ \,    t\in (0, T], \ x\in (-h_0, h_0),\\[3mm]
  u(t,x)=0  \quad  &{\rm for}\ \,  t\in (0, T],\ x\in \R\setminus  (-h_0, h_0),\\[3mm]
u(0, x)=u_0(x)  & {\rm for}\ \, x\in\R,
\end{array}\right.
$$

Then it follows from the maximum principle that $v_*\leq \bar u_n\leq w_*$ for all $n\geq 1$, which implies $v_*\leq \tilde u\leq w_*$. Since $v_*(t,x)\to u_0(x)$ and $w_*(t,x)\to u_0(x)$ in $L^\infty(\R)$ as $t\to 0^+$,
we deduce
\begin{equation}\label{t-to-0}
u_n(t,x)=e^{-L_0 t}\bar u_n(t,x)\to u_0(x) \mbox{ and $\tilde u(t,x)\to u_0(x)$ in $L^\infty(\R)$ as $t\to 0^+$, uniformly in $n$.}
\end{equation}
Therefore, for any given small $\epsilon>0$, we can find $\delta>0$, depending on $\epsilon$, so that
\begin{equation}\label{near-t=0}
|u_n(t,x)-\tilde u(t,x)|\leq \epsilon\ \,  \mbox{ for } t\in (0,\delta],\ x\in\R,\ n\geq 1.
\end{equation}

Denote
\[
\rho_n(t, x):=\begin{cases}\min\{h_n(t)-x, x-g_n(t)\} & \mbox{ if } x\in (g_n(t), h_n(t)),\\
0 &\mbox{ otherwise}, \end{cases} \]
\[
 \rho_*(t,x):=\begin{cases} \min\{h_*(t)-x, x-g_*(t)\} & \mbox{ if } x\in (g_*(t), h_*(t)),\\
0 &\mbox{ otherwise}. \end{cases}
\]
By \eqref{ab-2}, we have
\[
\displaystyle 0\leq u_n(t,x)\leq \frac{C_0}{\sqrt{\delta}}\rho_n^s(t, x)\ \,  \mbox{ for } t\in [\delta, T],\ x\in \R.
\]
Letting $n\to\infty$, we obtain
\begin{equation}\label{tilde-u-bdy}
\displaystyle 0\leq \tilde u(t, x)\leq \frac{C_0}{\sqrt{\delta}}\rho_*^s(t, x)\ \,  \mbox{ for } t\in [\delta, T],\ x\in \R.
\end{equation}
For any small $\hat \epsilon>0$, we can find $N=N(\hat \epsilon)>0$ large, depending on $\hat\epsilon$, so that
\[
|g_n(t)-g_*(t)|+|h_n(t)-h_*(t)|\leq \hat\epsilon\ \,  \mbox{ for } t\in [0, T],\ n\geq N.
\]
It follows that
\[
\rho_n(t,x)\leq \rho_*(t,x)+\hat\epsilon\ \, \mbox{ for } t\in [0, T],\  x\in\R,\ n\geq N.
\]
Therefore, for $t\in [\delta, T]$ and $x\in\R\setminus (g_*(t)+\hat\epsilon, h_*(t)-\hat\epsilon)$, we have
\[
\begin{cases}
\displaystyle 0\leq u_n(t,x)\leq \frac{C_0}{\sqrt{\delta}}\rho_n^s(t, x)\leq \frac{C_0}{\sqrt{\delta}}(2\hat\epsilon)^s\\[3mm]
\displaystyle 0\leq \tilde u(t, x)\leq \frac{C_0}{\sqrt{\delta}}\rho_*^s(t, x) \leq \frac{C_0}{\sqrt{\delta}} (\hat\epsilon)^s.
\end{cases}
\]
  By choosing $\hat\epsilon>0$ small enough, we deduce from the above estimates that
  \begin{equation}\label{near-bdy}
  |u_n(t,x)-\tilde u(t,x)|\leq\epsilon \mbox{ for  $t\in [\delta, T]$,\  $x\in\R\setminus (g_*(t)+\hat\epsilon, h_*(t)-\hat\epsilon)$ and $n\geq N(\hat\epsilon)$.}
  \end{equation}
  For $t\in [\delta, T]$ and $x\in [g_*(t)+\hat\epsilon, h_*(t)-\hat\epsilon]$, due to $u_n\to \tilde u$ in $C_{\loc}^{1+\frac{\alpha'}{2s}, 2s+\alpha'}(\Omega_{g_*,h_*, T})$, we can find $\tilde N=\tilde N(\delta,\hat\epsilon)>0$ large so that, for such $t$ and $x$,
  \[
  |u_n(t,x)-\tilde u(t,x)|\leq\epsilon\ \,  \mbox{ for } n\geq \tilde N.
  \]
  This, combined with \eqref{near-t=0} and \eqref{near-bdy}, shows that
  \[
  |u_n(t,x)-\tilde u(t,x)|\leq\epsilon\ \, \mbox{ for } t\in [0, T],\ x\in\R,
  \]
  provided that $n\geq \max\{N(\hat\epsilon), \tilde N(\delta, \hat\epsilon)\}$. We have thus proved $u_n\to \tilde u$ in $L^\infty([0,T]\times \R)$.
  From $u_n\to \tilde u$ in $C_{\loc}^{1+\frac{\alpha'}{2s}, 2s+\alpha'}(\Omega_{g_*,h_*, T})$ we know that $\tilde u$ satisfies \eqref{eq-u} with $(g,h)$ replaced by $(g_*, h_*)$ point-wisely in $\Omega_{g_*,h_*, T}$. By \eqref{t-to-0} we know that $\tilde u$ satisfies the initial condition in the classical sense, and it follows from \eqref{tilde-u-bdy}  that $\tilde u$ also satisfies the boundary condition in the classical sense. Therefore $\tilde u=u_*$ by uniqueness of the solution to \eqref{eq-u}. The proof is now complete.
   \end{proof}

 \begin{theorem}[Existence]\label{thm-ex}
 Problem \eqref{eq 1.1} has at least one classical solution.
 \end{theorem}
 \begin{proof}  It suffices to show that for any given $T>0$, there is a solution for $t\in (0, T]$.

 Fix $T>0$, we are going to reduce the question to a fixed point problem.
 For $R>0$  to be determined, we define
 \[\mathcal D_T(R):=\{(g,h)\in\bG_{h_0,T}\times \bH_{h_0, T}: \|g\|_\infty+\|h\|_\infty \leq R\}.
 \]
 Clearly $\mathcal D_T(R)$ is a bounded closed convex set in $C([0,T])\times C([0,T])$.
 For each $(g,h)\in\mathcal D_T(R)$, by Theorem \ref{thm-semi}, the problem \eqref{eq-u} has a unique solution $u=u^{g,h}$. We then define, for $t\in [0, T]$,
 \[
 \mathcal A(g,h)(t)=(\bar g(t), \bar h(t))
 \]
 with
\[\begin{cases}
\displaystyle  \bar h(t):= h_0+\mu \int_0^t \int_{g(\tau)}^{h(\tau)} u^{g,h}(\tau,x)[h(\tau )-x]^{-2s} dxd\tau,\\[4mm]
\displaystyle \bar g(t):= -h_0-\mu \int_0^t\int_{g(\tau)}^{h(\tau)} u^{g,h}(\tau,x)[x-g(\tau )]^{-2s} dxd\tau.
\end{cases}
\]
Clearly $\mathcal A(g, h)\in \bG_{h_0, T}\times \bH_{h_0, T}$. We are going to show that for suitably chosen  $R$, the operator $\mathcal A$ maps $\mathcal D_T(R)$ into itself and is completely continuous, which would guarantee the existence of a fixed point $(g^*, h^*)$ by Schauder's fixed point theorem, giving rise to a solution $(u,g,h)=(u^{g^*\!\!, h^*}, g^*, h^*)$ of \eqref{eq 1.1}.

By \eqref{ab-2},
\[
u^{g,h}(t,x)\leq \frac{C_0}{\sqrt{t}}\, \rho^s(t,x) \mbox{ for } t\in (0, T],\ x\in (g(t), h(t)).
\]
It follows that
\begin{align*}
\bar h(t)&\leq h_0+\mu \int_0^t \int_{g(\tau)}^{h(\tau)} \frac{C_0}{\sqrt{\tau}}[h(\tau)-x]^{-s} dxd\tau\\
&=h_0+\mu \int_0^t \frac{C_0}{\sqrt{\tau}}\frac{[h(\tau)-g(\tau)]^{1-s}}{1-s} d\tau\\
&\leq h_0+\mu \frac{[h(t)-g(t)]^{1-s}}{1-s}\int_0^t  \frac{C_0}{\sqrt{\tau}}d\tau  \\
&=h_0+\hat C_0 \sqrt{t}\, [h(t)-g(t)]^{1-s} \mbox{ for } t\in (0, T],
\end{align*}
with $\hat C_0>0$ depending only on $\mu$, $\|u_0\|_\infty$, $f$, $s$ and $T$.

Similarly,
\[
\bar g(t)\geq -h_0-\hat C_0 \sqrt{t}\, [h(t)-g(t)]^{1-s} \mbox{ for } t\in (0, T].
\]
Therefore,
\[
\bar h(t)-\bar g(t)\leq 2h_0+2\hat C_0 \sqrt{T}\, [h(t)-g(t)]^{1-s} \mbox{ for } t\in (0, T].
\]

We now  let $R=R_T>0$ be the unique solution of
\[
R= 2h_0+2\hat C_0 \sqrt{T} R^{1-s}.
\]
Then  $h(t)-g(t)\leq R$ implies $\bar h(t)-\bar g(t)\leq R$. Since $\|g\|_\infty+\|h\|_\infty=h(T)-g(T)$,
we thus see $\mathcal A$ maps $\mathcal D_T(R)$ into itself with the above chosen $R$.

We next show that $\mathcal A$ is continuous in $\mathcal D_T(R)$. Suppose $(g_*,h_*), (g_n, h_n)\in \mathcal D_T(R)$ for $n=1,2,...$, and $(g_n, h_n)\to (g_*, h_*)$ in $C([0, T])\times C([0, T])$ as $n\to\infty$.  We need to show that $(\bar h_n, \bar g_n):=\mathcal A(g_n, h_n)\to (\bar g_*, \bar h_*):=\mathcal A(g_*, h_*)$ in $C([0, T])\times C([0, T])$.

Let $u_*$ and $u_n$ be the solution of \eqref{eq-u} with $(g, h)=(g_*, h_*)$ and $(g,h)=(g_n, h_n)$, respectively.  By Lemma \ref{u-c}, we have $u_n\to u_*$ in $L^\infty([0,T]\times\R)$ as $n\to\infty$. It follows that, for any $\delta>0$ small,
\begin{equation}\label{delta-in}
\begin{cases}
\displaystyle \mu \int_0^t \int_{g_n(\tau)+\delta}^{h_n(\tau)-\delta } u_n(\tau,x)[h_n(\tau )-x]^{-2s} dxd\tau\to \mu \int_0^t \int_{g_*(\tau)+\delta}^{h_*(\tau)-\delta } u_*(\tau,x)[h_*(\tau )-x]^{-2s} dxd\tau,\\[4mm]
\displaystyle \mu \int_0^t \int_{g_n(\tau)+\delta}^{h_n(\tau)-\delta } u_n(\tau,x)[x-g_n(\tau )]^{-2s} dxd\tau\to \mu \int_0^t \int_{g_*(\tau)+\delta}^{h_*(\tau)-\delta } u_*(\tau,x)[x-g_*(\tau )]^{-2s} dxd\tau,
\end{cases}
\end{equation}
uniformly for $t\in [0, T]$ as $n\to+\infty$.

Define $\rho_n(t,x)$ and $\rho_*(t,x)$ as in the proof of Lemma \ref{u-c}.
By \eqref{ab-2},
\[
u_n(t,x)\leq \frac{C_0}{\sqrt{t}}\, \rho_n^s(t,x) \mbox{ for } t\in (0, T],\ x\in (g_n(t), h_n(t)).
\]
It follows that
\begin{align*}
&\quad \mu \int_0^t \int^{h_n(\tau)}_{h_n(\tau)-\delta } u_n(\tau,x)[h_n(\tau )-x]^{-2s} dxd\tau\\
&\leq \mu \int_0^t \int_{h_n(\tau)-\delta}^{h_n(\tau)} \frac{C_0}{\sqrt{\tau}}[h_n(\tau)-x]^{-s} dxd\tau\\
&=\frac{2\mu C_0}{1-s}\sqrt{t} \delta^{1-s}\leq \frac{2\mu C_0}{1-s}\sqrt{T} \delta^{1-s}\ \,   \mbox{ for all $t\in [0, T]$.}
\end{align*}
 Similarly,
\begin{align*}
&\quad \mu \int_0^t \int_{g_n(\tau)}^{g_n(\tau)+\delta } u_n(\tau,x)[x-g_n(\tau )]^{-2s} dxd\tau\\
&\leq \mu \int_0^t \int^{g_n(\tau)+\delta}_{g_n(\tau)} \frac{C_0}{\sqrt{\tau}}[x-g_n(\tau)]^{-s} dxd\tau\\
&=\frac{2\mu C_0}{1-s}\sqrt{t} \delta^{1-s}\leq \frac{2\mu C_0}{1-s}\sqrt{T} \delta^{1-s}\ \,    \mbox{ for all $t\in [0, T]$,}
\end{align*}
and
\[
\begin{cases}
\displaystyle\mu \int_0^t \int^{h_*(\tau)}_{h_*(\tau)-\delta } u_*(\tau,x)[h_*(\tau )-x]^{-2s} dxd\tau\leq \frac{2\mu C_0}{1-s}\sqrt{T} \delta^{1-s},\\[4mm]
\displaystyle\mu \int_0^t \int_{g_*(\tau)}^{g_*(\tau)+\delta } u_*(\tau,x)[x-g_*(\tau )]^{-2s} dxd\tau\leq \frac{2\mu C_0}{1-s}\sqrt{T} \delta^{1-s}
\end{cases}  \mbox{ for } t\in [0, T].
\]
These estimates and \eqref{delta-in} imply
\[\begin{cases}
\displaystyle \limsup_{n\to+\infty} \bar h_n(t)\leq \bar h_*(t)+2\frac{2\mu C_0}{1-s}\sqrt{T} \delta^{1-s},\\[2mm]
\displaystyle  \liminf_{n\to+\infty} \bar h_n(t)\geq \bar h_*(t)-2\frac{2\mu C_0}{1-s}\sqrt{T} \delta^{1-s},
\end{cases}\]
and the limits are uniform in $t\in [0, T]$.
Since $\delta>0$ can be arbitrarily small, this implies
\[
\lim_{n\to+\infty} \bar h_n(t)= \bar h_*(t) \mbox{ uniformly for } t\in [0, T].
\]
Analogously we have
\[
\lim_{n\to+\infty} \bar g_n(t)= \bar g_*(t) \mbox{ uniformly for } t\in [0, T].
\]
The continuity of $\mathcal A$ is now proved.

It remains to show that $\mathcal A$ is compact. Let $\{(g_n, h_n)\}\subset \mathcal D_T(R)$ be any sequence. We want to show that  $(\bar h_n, \bar g_n):=\mathcal A (g_n, h_n)$ has a convergent subsequence in $C([0, T])\times C([0, T])$. We have
\[
\bar h_n'(t)=\mu  \int_{g_n(t)}^{h_n(t)} u_n(t,x)[h_n(t)-x]^{-2s} dx>0 \mbox{ for } t\in (0, T],
\]
where $u_n:=u^{g_n, h_n}$. Since by \eqref{ab-2},
\[
u_n(t,x)\leq \frac{C_0}{\sqrt{t}}\, \rho_n^s(t,x) \mbox{ for } t\in (0, T],\ x\in (g_n(t), h_n(t)),
\]
we have
\[
0<\bar h'_n(t)\leq \mu  \int_{g_n(t)}^{h_n(t)} \frac{C_0}{\sqrt{t}}[h_n(t)-x]^{-s} dx=\frac{\mu C_0}{1-s}[h_n(t)-g_n(t)]^{1-s}/\sqrt{t}\leq \hat C_0/\sqrt{t}
\]
for $t\in (0, T]$, where $\hat C_0:=\frac{\mu C_0}{1-s} R^{1-s}$. Hence
\[
0\leq \bar h_n(t_1)-\bar h_n(t_2)\leq 2\hat C_0(\sqrt{t_1}-\sqrt{t_2}) \mbox{ for } 0\leq t_2\leq t_1\leq T.
\]
This implies that $\{\bar h_n(t)\}$ is equicontinuous over $[0, T]$. Similarly $\{\bar g_n(t)\}$ is equicontinuous in $[0, T]$.
By the Arzela-Ascoli  theorem, we easily see that $\{(\bar g_n, \bar h_n)\}$ has a convergent subsequence in $C([0, T])\times C([0, T])$, as desired.
 \end{proof}

 \subsection{Uniqueness} In this subsection, we prove the uniqueness part of Theorem \ref{teo 1}.
 We will need two lemmas. While the first is a rather natural comparison result,  the second, however, introduces a novel entropy (conservation) function, which will enable us to give a very special uniqueness proof, strikingly different from all the existing ones for related problems (where the approach is  based on the contraction mapping theorem).

 \begin{lemma}\label{c-p} Suppose that $(u_1,g_1,h_1)$ and $(u_2, g_2, h_2)$ are solutions of \eqref{eq 1.1} with $(u_0, h_0)=(u_0^1, h_0^1)$ and
 $(u_0, h_0)=(u_0^2, h_0^2)$, respectively, where $u_0^i\in \mathcal I_0((-h_0^i, h_0^i))$, $i=1,2$. If $h_0^1<h_0^2$ and $u_0^1\leq u_0^2$, then
 \[
 \mbox{$[g_1(t), h_1(t)]\subset (g_2(t), h_2(t))$ and
 $u_1(t,x)<u_2(t,x)$ for $t\in (0, T]$, $x\in [g_1(t), h_1(t)]$.}
 \]
  \end{lemma}

 \begin{proof}
 Since $h_0^1<h_0^2$, clearly $[g_1(t), h_1(t)]\subset (g_2(t), h_2(t))$ for all small $t>0$. We claim that
 \begin{equation}\label{g-h}
 [g_1(t), h_1(t)]\subset (g_2(t), h_2(t)) \mbox{  for all $t\in (0, T]$.}
 \end{equation}
  Otherwise, there exists $t_0\in (0, T]$ such that $[g_1(t), h_1(t)]\subset (g_2(t), h_2(t))$ for $t\in (0, t_0)$, and $g_1(t_0)=g_2(t_0)$ or $h_1(t_0)=h_2(t_0)$. For definiteness, we assume that $h_1(t_0)=h_2(t_0)$ and then
 \begin{equation}\label{t0}
 h_1'(t_0)\geq h_2'(t_0).
 \end{equation}
 Since $(g_1(t), h_1(t))\subset (g_2(t), h_2(t))$ for $t\in [0, t_0]$, and $u_0^1\leq u_0^2$, by applying the maximum principle to $u_2-u_1$ over $\Omega_{g_1, h_1, t_0}$ we deduce
 \[ u_1(t,x)<u_2(t,x) \mbox{ for } t\in (0, t_0],\ x\in (g_1(t), h_1(t)).
 \]
 It follows that
 \begin{align*}
 \displaystyle  h_2'(t_0)&=\mu  \int_{g_2(t_0)}^{h_2(t_0)} u_2(t_0,x)[h_2(t_0)-x]^{-2s} dx\\
 &=\mu  \int_{g_2(t_0)}^{h_1(t_0)} u_2(t_0,x)[h_1(t_0)-x]^{-2s} dx\\
 &\geq \mu  \int_{g_1(t_0)}^{h_1(t_0)} u_2(t_0,x)[h_1(t_0)-x]^{-2s} dx\\
 &>\mu  \int_{g_1(t_0)}^{h_1(t_0)} u_1(t_0,x)[h_1(t_0)-x]^{-2s} dx=h_1'(t_0),
 \end{align*}
 which is a contradiction to \eqref{t0}.
 This proves \eqref{g-h}, which allows us to apply the maximum principle to $u_2-u_1$ over $\Omega_{g_1, h_1, T}$ to deduce
 \[ u_1(t,x)<u_2(t,x) \mbox{ for } t\in (0, T],\ x\in (g_1(t), h_1(t)).
 \]
The proof is complete.
 \end{proof}
 \begin{remark}\label{rm-cp}
  From the above proof, we easily see that the assumptions that $(u_1,g_1,h_1)$ and $(u_2, g_2, h_2)$ are solutions of \eqref{eq 1.1} in Lemma \ref{c-p} can be relaxed to
 \[
\left\{
\begin{array}{lll}
\partial_t   u_1  +d(-\Delta)^s  u_1\leq  f(t,x, u_1)  \quad &{\rm for}\ \,   t\in(0,T],\ x\in\big(g_1(t),h_1(t)\big),\\[3mm]
 u_1(t,x)=0  \quad  &{\rm for}\ \, t\in(0,T],\ \, x\in \R \setminus\big(g_1(t),h_1(t)\big),\\[4mm]
\displaystyle  h_1'(t)\leq \mu \int_{g_1(t)}^{h_1(t)} u_1(t,x)[h_1(t)-x]^{-2s} dx  \ \ &{\rm for}\ \,t\in(0,T],\\[4mm]
\displaystyle g_1'(t)\geq -\mu \int_{g_1(t)}^{h_1(t)} u_1(t,x)[x-g_1(t)]^{-2s} dx   \ \  &{\rm for}\ \,t\in(0,T],\\[4mm]
 u_1(0,x)=u^1_0(x)  &{\rm for}\ \,  x\in [-h_0^1, h_0^1],\\[4mm]
 h_1(0)=-g_1(0)=h^1_0,
\end{array}\right.
 \]
 and
 \[
\left\{
\begin{array}{lll}
\partial_t   u_2  +d(-\Delta)^s  u_2\geq  f(t,x, u_2)  \quad &{\rm for}\ \,   t\in(0,T],\ x\in\big(g_2(t),h_2(t)\big),\\[3mm]
 u_2(t,x)\geq 0  \quad  &{\rm for}\ \, t\in(0,T],\ \, x\in \R \setminus\big(g_1(t),h_1(t)\big),\\[4mm]
\displaystyle  h_2'(t)\geq \mu  \int_{g_2(t)}^{h_2(t)} u_2(t,x)[h_2(t)-x]^{-2s} dx  \ \ &{\rm for}\ \,t\in(0,T],\\[4mm]
\displaystyle g_2'(t)\leq -\mu \int_{g_2(t)}^{h_2(t)} u_2(t,x)[x-g_2(t)]^{-2s} dx   \ \  &{\rm for}\ \,t\in(0,T],\\[4mm]
 u_2(0,x)=u^2_0(x)  &{\rm for}\ \,  x\in [-h_0^2, h_0^2],\\[4mm]
 h_2(0)=-g_2(0)=h^2_0,
\end{array}\right.
 \]
 provided that all the terms are well-defined. Moreover, the assumption $(g_i(0), h_i(0))=(-h_0^i, h_0^i)$ can be replaced by $(g_i(0), h_i(0))=\cO_i$ for any open intervals $\cO_i$ $(i=1,2)$,  as long as
 $\overline \cO_1\subset \cO_2$ and $u_0^1\leq u_0^2$ in $\cO_1$.
  \end{remark}

 Let $L_0$, $K_1$ be determined by ${\bf (f1)}$ as in \eqref{eq lam-0}.   Then
\begin{equation}\label{F1}
 \ F_1(t,x,\xi):=-L_0 \xi+   e^{-L_0 t}f(t,x, e^{L_0t} \xi)
 \end{equation}
  is decreasing in $\xi\in [0, e^{L_0T}K_1]$  for any fixed $(t,x)\in[0,T]\times \R$, and $F_1(t,x, 0)\equiv 0$.

  We now define, for any $(g,h)\in \bG_{h_0, T}\times \bH_{h_0, T}$ and
 any continuous function $u$ over $\overline \Omega_{g,h, T}$, the entropy function
 \begin{align*}
 \widehat\Psi_{(u,g,h)}(t):= &\int_{g(t)}^{h(t)}\!\!u(t,x)dx + \frac{d e^{-L_0t}}{\mu} \big(h(t)\!-\!g(t)\big)+\frac{dL_0}{\mu}\!\int_0^{t}\!\big(h(\tau)-g(\tau)\big) e^{-L_0\tau} d\tau\\
 & \ \ \
 -\int_0^{t}\!\!\int_{g(\tau)}^{h(\tau)} \!\! F_1(\tau,x,e^{-L_0\tau} u) dx d\tau.
 \end{align*}

 \begin{lemma}\label{lem-ent} The entropy function $\widehat \Psi_{(u,g,h)}(\cdot)$ is increasing in $u, h$ and $-g$ provided that $u\leq K_1$. Moreover, if
  $(u, g, h)$ is a solution of \eqref{eq 1.1}, then $\widehat \Psi_{(u,g,h)}(t)\equiv \widehat \Psi_{(u,g,h)}(0) $ for $t\in (0, T]$.
 \end{lemma}
 \begin{proof}
 The monotonicity of $\widehat \Psi_{(u,g,h)}(\cdot)$  in $u, h$ and $-g$ follows directly from its definition and the monotonicity of the nonpositive function $F_1$ in its third argument.

 It remains to prove the second assertion.
 Let $( u, g,h)$ be any solution of \eqref{eq 1.1}. Then
  \[
  \tilde u(t,x):=e^{-L_0t}u(t,x)
  \]
  satisfies
  \begin{equation}\label{eq-F1}
\left\{
\begin{array}{lll}
\partial_t  \tilde u  +d(-\Delta)^s  \tilde u= F_1(t,x,\tilde u)  \quad &{\rm for}\ \,   t\in(0,T],\ x\in\big(g(t),h(t)\big),\\[3mm]
 \tilde u(t,x)=0  \quad  &{\rm for}\ \, t\in(0,T],\ \, x\in \R \setminus\big(g(t),h(t)\big),\\[2mm]
\displaystyle  h'(t)=\mu e^{L_0 t} \int_{g(t)}^{h(t)}\tilde u(t,x)[h(t)-x]^{-2s} dx  \ \ &{\rm for}\ \,t\in(0,T],\\[4mm]
\displaystyle g'(t)=-\mu e^{L_0 t}\int_{g(t)}^{h(t)}\tilde u(t,x)[x-g(t)]^{-2s} dx   \ \  &{\rm for}\ \,t\in(0,T],\\[4mm]
 h(0)=-g(0)=h_0,\ \tilde u(0,x)=u_0(x)  &{\rm for}\ \,  x\in\cO_0.
\end{array}\right.
 \end{equation}

Define
$$\Psi(t):=\int_{g(t)}^{h(t)} \tilde u(t,x) dx=\int_{g(t)}^{h(t)} e^{-L_0t} u(t,x) dx.  $$
Then
  \begin{align*}
\Psi'(t)&=h'(t)\tilde u(t,h(t))-g'(t) \tilde u(t,g(t))+\int_{g(t)}^{h(t)}\partial_t \tilde u(t,x) dx
\\& =-d  \int_{g(t)}^{h(t)} (-\Delta)^s  \tilde u(t,x) dx +  \int_{g(t)}^{h(t)} F_1(t,x,\tilde u) dx.
\end{align*}
A direct computation shows that, with $\cO_t:=(g(t), h(t))$,
 \begin{align*}
  \int_{g(t)}^{h(t)} (-\Delta)^s   u(t,x) dx &= \lim_{\epsilon\to 0^+}\int_{\cO_t} \int_{\R\setminus B_\epsilon(x)} \frac{  u(t,x)-  u(t,y)}{|x-y|^{1+2s}} dy dx
 \\&=\lim_{\epsilon\to 0^+}\int_{\cO_t} \int_{\cO_t\setminus B_\epsilon(x)}  \frac{  u(t,x)-  u(t,y)}{|x-y|^{1+2s}} dy dx\\
 &\ \ \ \ +\int_{g(t)}^{h(t)} \int_{h(t)}^{+\infty} \frac{   u(t,x)-  u(t,y)}{|x-y|^{1+2s}} dy dx
 +\int_{g(t)}^{h(t)} \int_{-\infty}^{g(t)} \frac{  u(t,x)-  u(t,y)}{|x-y|^{1+2s}} dy dx
\\&=\int_{g(t)}^{h(t)}\int_{h(t)}^{+\infty} \frac{  u(t,x) }{|x-y|^{1+2s}} dy dx +\int_{g(t)}^{h(t)} \int_{-\infty}^{g(t)} \frac{   u(t,x) }{|x-y|^{1+2s}} dy dx
\\&= \frac{1}{\mu}\big(h'(t)-g'(t)\big),
\end{align*}
 since the anti-symmetry  in $(x,y)$ of the integrand  function below  gives
 \[
\int_{\cO_t} \int_{\cO_t\setminus B_\epsilon(x)}  \frac{  \ u(t,x)-  u(t,y)}{|x-y|^{1+2s}} dy dx
=\int\!\!\!\int_{(\cO_t\times \cO_t)\setminus\Delta_\epsilon(t)}   \frac{ u(t, x)-  u(t, y)}{|x-y|^{1+2s}} dxdy=0,
 \]
 where $\Delta_\epsilon(t):=\{(x,y)\in \cO_t\times \cO_t: |x-y|<\epsilon\}$.

 Thus
 \[
 d  \int_{g(t)}^{h(t)} (-\Delta)^s  \tilde u(t,x) dx=\frac{d e^{-L_0t}}{\mu}\big(h'(t)-g'(t)\big)
 \]
 and
 \begin{align*}
\Psi'(t) + \frac{d e^{-L_0t}}{\mu}\big(h'(t)-g'(t)\big)- \int_{g(t)}^{h(t)} F_1(t,x, e^{-L_0t} u) dx=0\quad {\rm for}\ \, t\in (0,T].
\end{align*}
It follows that $\widehat \Psi'(t)=0$ for any $t\in (0, T]$,
where
{\small
\begin{align*}\widehat \Psi (t):&= \Psi (t) \! +\! \frac{d e^{-L_0t}}{\mu} \big(h(t)\!-\!g(t)\big)+\frac{dL_0}{\mu}\int_0^{t}\big(h(\tau)-g(\tau)\big) e^{-L_0\tau} d\tau
 -\int_0^{t}\int_{g(\tau)}^{h(\tau)} F_1(\tau,x,e^{-L_0\tau} u) dx d\tau\\
 &=\widehat \Psi_{(u,g,h)}(t).
 \end{align*}
 }
 Hence
 \begin{align}\label{identity-0}
\widehat \Psi_{(u,g,h)}(t) = \widehat \Psi_{(u,g,h)}(0) =\int_{-h_0}^{h_0} u_0(x)dx+\frac d\mu (2h_0) \mbox{ for } t\in (0, T].
 \end{align}
 The proof is complete.
 \end{proof}

 \begin{theorem}[Uniqueness]\label{thm-uniq}
 Problem \eqref{eq 1.1} has a unique solution.
 \end{theorem}
 \begin{proof}
 We will prove the desired uniqueness result by a contradiction argument, where the entropy function $\widehat \Psi_{(u,g,h)}$ in Lemma \ref{lem-ent} will play a crucial role.

So suppose that $(u_1, g_1, h_1)$ and $(u_2, g_2, h_2)$ are two different solutions of \eqref{eq 1.1} defined for $t\in (0, T]$. By the uniqueness conclusion for \eqref{eq-u}, necessarily $(g_1, h_1)\not\equiv (g_2, h_2)$, and therefore, there exists $i\in \{1,2\}$ such that
\[
(u_i, g_i, h_i)\not\equiv (u_*, g_*, h_*), \ \mbox{ where } \begin{cases}g_*(t):=\min\{g_1(t), g_2(t)\},\ \\
h_*(t):=\max\{h_1(t), h_2(t)\},\\
 u_*(t,x):=\max\{u_1(t,x), u_2(t,x)\}.\end{cases}
\]
For definiteness, we assume $i=1$.
We now let $\widehat \Psi_i(t)$ denote $\widehat\Psi_{(u,g,h)}(t)$ with $(u,g,h)=(u_i, g_i, h_i)$, $i=1,2$, and similarly define $\widehat \Psi_*(t)$. By \eqref{bound-1}, we have $0\leq u_i\leq u_*\leq K_1$ for $i=1,2$.
Then from the  monotonicity property of the entropy function, and our assumption above that $(u_1, g_1, h_1)\not\equiv (u_*, g_*, h_*)$, we obtain
\begin{equation}\label{1<*}
\widehat \Psi_1(T)<\widehat \Psi_*(T).
\end{equation}

Next we choose a sequence $h_0^n>h_0$ decreasing to $h_0$ as $n\to\infty$, and a sequence $u_0^n\in\mathcal I_0((-h_0^n, h_0^n))$ such that
\[
u_0^n\geq u_0,\ \lim_{n\to+\infty} u_0^n= u_0 \mbox{ in } L^\infty(\R) .
\]
By Theorem \ref{thm-ex}, problem \eqref{eq 1.1} with initial data $(u_0^n, h_0^n)$ has a solution $( u^*_n, g^*_n, h^*_n)$  for each $n\geq 1$. By Lemma \ref{c-p} and \eqref{bound-1} we have
\[
[g_i(t), h_i(t)]\subset (g^*_n(t), h^*_n(t)),\ \  u_i(t,x)\leq u^*_n(t,x)\leq K_1 \mbox{ for } t\in (0, T],\ x\in \R, \ i=1,2,\ n\geq 1.
\]
It follows that
\[
[g_*(t), h_*(t)]\subset (g^*_n(t), h^*_n(t)),\ \  u_*(t,x)\leq u^*_n(t,x) \leq K_1 \mbox{ for } t\in (0, T],\ x\in \R.
\]
Hence
\[
\widehat \Psi_*(T)\leq \widehat \Psi_{n,*}(T) \mbox{ for  all } n\geq 1,
\]
where
$\widehat\Psi_{n,*}$ denotes $\widehat \Psi_{(u,g,h)}$ with $(u,g,h)=(u^*_n, g^*_n, h^*_n)$.
Since
\[\begin{aligned}
\widehat \Psi_{n,*}(T)=\widehat \Psi_{n,*}(0)&=\int_{-h^n_0}^{h^n_0} u^n_0(x)dx+\frac d\mu (2h^n_0)\\
& \to \int_{-h_0}^{h_0} u_0(x)dx+\frac d\mu (2h_0)=\widehat\Psi_1(0)=\widehat\Psi_1(T)
\mbox{ as } n\to\infty,
\end{aligned}
\]
it follows that $\widehat \Psi_*(T)\leq \widehat\Psi_1(T)$, which is a contradiction to \eqref{1<*}, and the proof is complete.
\end{proof}

   \setcounter{equation}{0}
\section{Spreading-vanishing dichotomy}

We examine the long-time dynamics of \eqref{eq 1.1} in this section.
\subsection{Fixed domain problem}

We start this section with the long time behavior of fractional parabolic equations in a cylinder domain
 \begin{equation}\label{eq 4.1}
\left\{ \arraycolsep=1pt
\begin{array}{cll}
\displaystyle \partial_t u+ d (-\Delta)^s  u=f(u)\quad \  &{\rm in}\ \   (0,+\infty)\times \bI,\\[2mm]
\displaystyle   u  = 0  \ &{{\rm in}}\  \    (0,+\infty)\times \big(\R \setminus \bI\big), \\[2mm]

  u(0,\cdot) =u_0   \ &{{\rm in}}\  \ \bI,
\end{array}
\right.
\end{equation}
where $\bI$ is a finite open interval in $\R$.  By Theorem \ref{thm-semi},  problem \eqref{eq 4.1} admits a unique positive solution $u_{\bI}$.

A function $v$ is called a steady state of \eqref{eq 4.1} if  $v\in C(\R)$ is a classical solution of
\begin{equation}\label{eeq 4.1}
	\left\{ \arraycolsep=1pt
	\begin{array}{cll}
		\displaystyle   d (-\Delta)^s v = f(v ) \quad \  &{\rm in}\ \    \bI,
		\\[2mm]
		\displaystyle  v=0  \ &{{\rm in}}\  \    \R \setminus \bI.
	\end{array}
	\right.
\end{equation}

Let $\lambda_{s,1}(\bI)$ be the principal eigenvalue of $(-\Delta)^s$ over $\bI$ with Dirichlet boundary conditions,
and $\phi_{s,1}=\phi_{s,1,\bI}$ an associated principal eigenfunction that is positive in $\bI$.

\begin{lemma}\label{lm 4.1-0}
Assume that  $f$ verifies ${\bf (f1)-(f4)}$ and $\bI$ is a bounded interval. Then the
problem \eqref{eeq 4.1} has only the trivial nonnegative solution $v\equiv 0$ if $d \lambda_{s,1}(\bI )\geq  f'(0)$,
and \eqref{eeq 4.1} admits a unique positive  solution if $d \lambda_{s,1}(\bI )<  f'(0)$.
\end{lemma}
\begin{proof} Let $v$ be a nonnegative solution of \eqref{eeq 4.1}. If  $d \lambda_{s,1}(\bI )\geq  f'(0)$, then
 \begin{align*}
d\lambda_{s,1} \int_{\bI}\phi_{s,1}   vdx=d \int_{\bI}   (-\Delta)^s \phi_{s,1} vdx &=d \int_{\bI}    \phi_{s,1}(-\Delta)^s  vdx
  \\[1mm]&= \int_{\bI}  \phi_{s,1}f(v)dx
\leq f'(0)   \int_{\bI}  \phi_{s,1} vdx,
  \end{align*}
with the last inequality strict unless $v\equiv0$. Thus $v\gneqq 0$ implies
$$\int_{\bI}  \phi_{s,1} vdx>0\quad{\rm and}\quad
0\leq \big(d\lambda_{s,1}-f'(0)\big)  \int_{\bI}  \phi_{s,1} vdx<0, $$
which is impossible. Hence \eqref{eeq 4.1} has no nontrivial nonnegative solution in this case.

If  $d \lambda_{s,1}(\bI )<  f'(0)$, then there exists $\delta_0\in(0,u_*)$ such that
$$f(\tau)\geq d \lambda_{s,1}(\bI )  \tau  \quad {\rm for}\ \, \tau\in[0,\delta_0].  $$
Thanks to the boundedness of $\phi_{s,1}$,  there exists $\epsilon_0>0$ such that $\epsilon \phi_{s,1}\leq  \delta_0$ for $\epsilon\in(0,\epsilon_0]$.
Let $v_1=\epsilon \phi_{s,1}$; then
$$d(-\Delta)^s v_1=d   \lambda_{s,1}(\bI ) v_1  \leq        f(v_1).$$
Let $v_i$ be the unique classical solution of
\begin{equation*}
\begin{cases}
\displaystyle   d (-\Delta)^s v +Kv= f(v_{i-1} )+Kv_{i-1} \quad \  &{\rm in}\ \    \bI,\\[2mm]
\phantom{ (-\Delta)^s\    }
\displaystyle  v=0 \quad \ &{{\rm in}}\  \    \R \setminus \bI,
\end{cases}
\end{equation*}
where $K>0$ is some large constant such that $f(\tau)+K\tau$ is increasing for $\tau\in [0,u_*]$, and consider the  sequence $\{v_i\}_{i=1}^{\infty}$. The comparison principle of the fractional Laplacian implies that $v_i$ is  increasing in terms of $i$ and bounded from above by $u_*>0$. The stability results \cite[Theorem 2.4]{CFQ} imply that the limit of $v_i$ as $i\to+\infty$ is a positive solution of \eqref{eeq 4.1}.

 It remains to  show  the uniqueness of positive solutions to \eqref{eeq 4.1}.  Let $v_1$ and $v_2$ be two classical positive solutions of \eqref{eeq 4.1}.
 By boundary regularity (see \cite{RS}) and the Hopf boundary lemma (see \cite{GS}) for $(-\Delta)^s$, there is a small constant $\epsilon_1>0$ such that $\tilde v_1:=\epsilon_1 \phi_{s,1}\leq \min\{v_1,v_2\}$. Repeating the above iteration arguments with $u_*$ replaced by $\tilde u_*:=\max\{\|v_1\|_\infty, \|v_2\|_\infty\}$, we obtain an increasing sequence $\{\tilde v_i\}_{i=1}^{\infty}$, and $\displaystyle v_*=\lim_{i\to \infty} \tilde v_i$ is a classical solution of \eqref{eeq 4.1} satisfying $v_*\leq \min\{v_1,v_2\}$.
 Indeed, by enlarging $K$ suitably we may assume that $f(\tau)+K\tau$ is increasing for $\tau\in [0, \tilde u_*]$.
Clearly $\tilde v_1\leq \min\{v_1,v_2\}$. If $\tilde v_{i-1}\leq \min\{v_1,v_2\}$,  then
\begin{align*}
d(-\Delta)^s (\tilde v_{i}-v_1)+K(\tilde v_{i}-v_1)\leq 0,\ \ d(-\Delta)^s (\tilde v_{i}-v_2)+K(\tilde v_{i}-v_2)\leq 0,
\end{align*}
which imply $\tilde v_i\leq \min\{v_1,v_2\}$ by the maximum principle. Hence, $\tilde v_{i}\leq \min\{v_1,v_2\}$, and we have proved that this inequality holds for all $i\geq 1$. It follows that $v_*\leq \min\{v_1,v_2\}$.

Now using $\displaystyle \int_{\R} v_*(-\Delta)^{s/2} v_1 {\rm d}x=\int_{\R} v_1(-\Delta)^{s/2} v_* {\rm d}x$ and \eqref{eeq 4.1} we obatin
 \begin{align*}
 0=\int_{\bI}[v_*f(v_1)-v_1f(v_*)] {\rm d}x=\int_{\bI}v_*v_1[f(v_1)/v_1-f(v_*)/v_*] {\rm d}x.
 \end{align*}
 Since $v_*\leq v_1$, by the  maximum principle we have 
 \[
 \mbox{either $v_*\equiv v_1$ or $v_*<v_1$ in $\bI$.}
 \]
   By \textbf{(f4)}, we have $f(v_1)/v_1-f(v_*)/v_*\leq 0$. Hence,
 \begin{align*}
 0= \int_{\bI}v_*v_1[f(v_1)/v_1-f(v_*)/v_*] {\rm d}x\leq 0,
 \end{align*}
which implies $f(v_1)/v_1-f(v_*)/v_*\equiv 0$ in $\bI$. Due to ${\bf (f4)}$, this is possible only if the alternative $v_1\equiv v_*$ above holds.  Similarly,  one could show that $v_2\equiv v_*$. Hence $v_1\equiv v_2$.
\end{proof}

\begin{proposition} \label{pr cyl-1}
  Suppose that  $f$ verifies ${\bf (f1)-(f4)}$ and $u_0  \in  \mathcal I_0 (\bI)$. Then the unique positive solution of  \eqref{eq 4.1}, denoted by $u_{\bI}(t,x)$, has the following properties:

 \begin{itemize}
 	\item[{\rm (i)}] If $d \lambda_{s,1}(\bI )\geq  f'(0)$,   then $u_{\bI}(t,\cdot)\to 0$ in $C(\bar\bI)$ as $t\rightarrow +\infty$;
 	\item[{\rm (ii)}] If $d\lambda_{s,1}(\bI )< f'(0)$, then, as $t\rightarrow +\infty$, $u_{\bI}(t,\cdot)$ converges  to the unique positive   steady state $\hat u_{\bI}$ in $C^{{2s+\gamma}}_{\loc}(\bI)\cap C(\bar\bI)$ for some $\gamma\in(0,s)$.	

 	\item[{\rm (iii)}]  Let $\bI=(\ell_1,\ell_2)$, then
 	$$\lim_{\ell_1\to-\infty,\, \ell_2\to\infty}  \hat u_{\bI} =u_*\quad {\rm locally\ uniformly\ in\ }\, \R, $$
 	where $u_*>0$ is the positive zero point of $f$, i.e. $u_*>0,\ f(u_*)=0$.
 \end{itemize}
\end{proposition}
\begin{proof}     (i) Define
$$ W_0(t,x):=c_0 e^{-\big(d\lambda_{s,1}(\bI)-f'(0)\big) t } \phi_{s,1}(x),\quad (t,x)\in(0,+\infty)\times \R,    $$
where
$$c_0:=\sup_{x\in\bI}\frac{u_{\bI}(1,x)}{\phi_{s,1}(x)}$$
is a finite positive constant due to the boundary regularity for $u_{\bI}$ and the Hopf lemma for $\phi_{s,1}$.

Notice that for $(t,x)\in[1,+\infty)\times \bI,$
\begin{align}\label{e 4.1-1}
   \partial_t W_0(t,x)+ d (-\Delta)^s   W_0(t,x)-f( W_0(t,x))
   = f'(0) W_0(t,x)-f(  W_0(t,x))
   \geq 0,
\end{align}
since $f( W_0)\leq f'(0) W_0$. By the maximum principle we obtain $0\leq u_{\bI}(1+t,x) \leq W_0(t,x)$ for $(t,x)\in (0,+\infty)\times\R$.

 If $d \lambda_{s,1}(\bI )>  f'(0)$, we see from the definition of $W_0$ that $W_0(t,x)\to 0$ uniformly in $x\in\R$ as $t\to\infty$, and it follows that
$$\lim_{t\to+\infty}u_{\bI}(t,x)=0\quad{\rm uniformly\ in}\ \R. $$

We next consider the case $d \lambda_{s,1}(\bI )=  f'(0)$. In this case,
let $  U_0(t,x)$ be the solution of
\begin{equation}\label{eq 4.1-up}
\left\{ \arraycolsep=1pt
\begin{array}{lll}
\displaystyle \partial_t u+ d (-\Delta)^s  u=f(u)\quad \  &{\rm in}\ \   (0,+\infty)\times \bI,\\[2mm]
 \phantom{ (-\Delta)^s --\ \,   }
\displaystyle   u  = 0 \quad \ &{{\rm in}}\  \    ( 0,+\infty)\times \big(\R \setminus \bI\big), \\[2mm]
 \phantom{ (-\Delta)^s  \   }
  u(0,\cdot) = c_0\phi_{s,1}  \quad \ &{{\rm in}}\  \ \bI.
\end{array}
\right.
\end{equation}
 Since $W_0$ is a super solution of \eqref{eq 4.1-up}, we have, by the maximum principle,
$$ 0\leq U_0(t,x)\leq W_0(t,x)= c_0   \phi_{s,1}(x)=U_0(0,x)\quad {\rm for}\ \, t>0. $$
As a consequence,  for any $\epsilon>0$,
$$ U_0(t+\epsilon,x)\leq U_0(t,x)\quad {\rm for}\  t>0, $$
which implies that
the mapping $t\mapsto U_0(t,\cdot)$ is non-increasing and nonnegative.  This ensures that the limit of $U_0(t,\cdot)$ exists   as $t\to+\infty$. Denote
$$\hat{U}_0(x):=\lim_{t\to+\infty} U_0(t,x). $$

{\bf Claim:} $\hat{U}_0(x)\equiv 0$.

By \cite[Corollary 1.6]{FR2017} and \cite[Corollary 1.2]{RV}, for given $\gamma\in (0,s)$ and small $\epsilon>0$, there exists $C$ depending on $||U_0||_{L^\infty}$ and $||f(U_0)||_{L^\infty}$ such that
\begin{align}\label{4.7a}
&||\partial_{t} U_0||_{C_{t,x}^{\gamma/2s,\gamma}}((n,n+2)\times \bar \bI)+||U_0/d^s||_{C_{t,x}^{1/2+\gamma/2s,\gamma+s}}((n,n+2)\times \bar \bI)\nonumber\\
&+||U_0||_{C_{t,x}^{1-\epsilon,s}}((n,n+2)\times \bar \bI)\leq C
\end{align}
for all $ n=1,2,\cdots$.
Thus, $\partial_{t} U_0$ is uniformly continuous for $t\in [0,\infty)$. This, combined with Barbalat's lemma  and the convergence of $U_0$ as $t\to+\infty$, implies
\begin{align*}
\lim_{t\to +\infty}\partial_{t} U_0(t,x)=0,\ \ \forall x\in \bar \bI.
\end{align*}
Moreover, it follows from \eqref{4.7a} and the convergence of $U_0$ to $\hat U_0$  that $\hat{U}_0/d^s\in C^{\gamma+s-\epsilon_1}( \bar \bI)$, $\hat{U}_0\in C^{s-\epsilon_1}( \bar \bI)$ and
\begin{align*}
\lim_{t\to +\infty}(-\Delta)^s  U_0(t,x)=(-\Delta)^s \hat{U}_0(x),\ \  \ \lim_{t\to +\infty} f(U_0(t,x))=f(\hat{U}_0(x)), \ \ \ \ \forall\, x\in \bar\bI.
\end{align*}
Hence we can let $t\to\infty$ in \eqref{eq 4.1-up} to deduce that $\hat{U}_0(x)\geq 0$ is a solution of \eqref{eeq 4.1}.  By Lemma \ref{lm 4.1-0}, $\hat{U}_0(x)\equiv 0$ for all $x\in \R$, and the above claim is proved.
Moreover,
\begin{equation}\label{e 4.2}
\lim_{t\to+\infty}\|U_0(t,\cdot)\|_{L^\infty(\R)}=0.
\end{equation}
Since $0\leq u(1,x)\leq c_0\phi_{s,\bI,1}$, by the maximum principle we have $0\leq u(t+1, x)\leq U_0(t,x)$ for $t>0$ and $x\in\R$.
Thus (\ref{e 4.2}) implies that  $u_{\bI}(t,x)$ converges to  $0$ in $L^\infty(\R)$ as $t\rightarrow +\infty$.\medskip

 (ii) For $d \lambda_{s,1}(\bI )<  f'(0)$,  let $\bu  (t,x;  \tilde u_0)$ be the solution of
\begin{equation}\label{eq 4.1-up-ep}
\left\{ \arraycolsep=1pt
\begin{array}{lll}
\displaystyle \partial_t u+ d (-\Delta)^s  u=f(u)\quad \  &{\rm in}\ \   (0,+\infty)\times \bI,\\[2mm]
 \phantom{ (-\Delta)^s --\ \ \ \,  }
\displaystyle   u  = 0 \quad \ &{{\rm in}}\  \    ( 0,+\infty)\times \big(\R \setminus \bI\big), \\[2mm]
 \phantom{ (-\Delta)^s  \ \ \   }
  u(0,\cdot) = \tilde u_0  \quad \ &{{\rm in}}\  \ \bI
\end{array}
\right.
\end{equation}
for some $\tilde u_0\in \mathcal I_0(\bI)$. Taking $\tilde u_0=\epsilon \phi_{s,1}$,
direct computations show that for small $\epsilon>0$,
 \begin{align*}
 \partial_t \tilde u_0 + d (-\Delta)^s  \tilde u_0 -f( \tilde u_0)
  =    d\lambda_{s,1}(\bI) \phi_{s,1}-f(\epsilon   \phi_{s,1})
  \leq  0,
 \end{align*}
since $f'(0)>d\lambda_{s,1}(\bI)$ and $f(\tau)>\tau d\lambda_{s,1}(\bI) $ for small $\tau$.
Then the maximum principle implies  $\bu(t,x; \tilde u_0)\geq  \tilde u_0(x)=\bu(0,x;\tilde u_0)$ for $t>0$,
which implies that  the mapping $t\mapsto  \bu(t,\cdot; \tilde u_0)$ is non-decreasing. Clearly,  $u_*$ is a super solution for \eqref{eq 4.1-up-ep}, and so $\bu\leq u_*$. By a similar discussion as in (i) and using Lemma \ref{lm 4.1-0} again, we could show that there exists a  positive function  $\hat{u}_{\bI}\in C^{s-\epsilon_1}( \bar \bI)$ such that $\hat{u}_{\bI}(x)$ is the unique positive solution of \eqref{eeq 4.1}, and  $\hat{u}_{\bI}/d^s\in C^{\gamma+s-\epsilon_1}( \bar \bI)$,
 \begin{equation}\label{4.10}
	\lim_{t\to+\infty}
	\| \bu(t,\cdot; \tilde u_0)- \hat{u}_{\bI} \|_{C(\bar \bI)}=0.
\end{equation}

 Let $\eta_{\bI}$ be the solution of
\begin{equation*}
	\left\{ \arraycolsep=1pt
	\begin{array}{lll}
		\displaystyle   d (-\Delta)^s v=  1 \quad \  &{\rm in}\ \    \bI,\\[2mm]
		\phantom{ (-\Delta)^s\ \,   }
		\displaystyle  v=0 \quad \ &{{\rm in}}\  \    \R \setminus \bI.
	\end{array}
	\right.
\end{equation*}
For a given large constant $k$ such that $ k>\max_{\xi>0} f(\xi)$, we derive
\begin{align*}
d (-\Delta)^s  (k \eta_{\bI}(x)) -f   (k \eta_{\bI}(x))
	=   k-f(  k \eta_{\bI}(x)) \geq  0.
\end{align*}
Similarly to the above discussion, we can show that
the mapping $t\mapsto  \bu(t,\cdot\,; k \eta_{\bI})$ is decreasing and positive, and
\begin{align}\label{4.11}
\lim_{t\to+\infty}
\| \bu(t,\cdot\, ; k \eta_{\bI})- \hat{u}_{\bI} \|_{C(\bar \bI)}=0.
\end{align}

By taking $\epsilon>0$ small enough and $k>0$ large enough, we have
\[
k \eta_{\bI}\geq u(1,\cdot)\geq \epsilon\phi_{s,1}.
\]
It then follows from the comparison principle that
\[
\bu(t,\cdot\,; k \eta_{\bI})\geq u(t+1,\cdot) \geq \bu(t,\cdot\, ; \epsilon\phi_{s,1}) \mbox{ for } t>0.
\]
  The desired conclusion now follows immediately from  \eqref{4.10} and  \eqref{4.11}.
\smallskip

(iii)  We see from the proof of (ii) above  that
 \begin{align}\label{bon-1}
\hat{u}_{\bI}\leq u_*.
   \end{align}

If $\bI_1\subset \bI_2$, then clearly $\eta_{\bI_1}\leq \eta_{\bI_2}$ and
thus $\bu(t,x; k \eta_{\bI_1})\leq \bu(t,x; k \eta_{\bI_2})$ for any $(t,x)\in (0,+\infty)\times \R$,
which implies that
 \begin{align}\label{mon-1}
 \hat{u}_{\bI_1}\leq\hat{u}_{\bI_2}\leq u_*\quad {\rm in}\ \R.
   \end{align}
 This shows that the mapping $\bI\mapsto \hat{u}_{\bI}$ is non-decreasing. Taking $\bI_n=(-n,n)$, then
 \begin{align*}
 \hat{u}_{\R}(x):=\lim_{n\to \infty} \hat{u}_{\bI_n}(x)\ {\rm is\ well\ defined}.
 \end{align*}

\noindent{\bf Claim:}  $\hat{u}_{\R}(x)$ is a constant.

 If $x_2-1<x_1<x_2$ and
  \begin{align*}
  \bI_{2n+1}^*:= \bI_{2n+1}-x_1+x_2:=\{x+x_1-x_2:x\in \bI_{2n+1} \}.
 \end{align*}
 Then $\hat{u}_{\bI_{2n+1}}(x_2)= \hat{u}_{\bI_{2n+1}^*}(x_1)$. Noting that $\bI_{2n}\subset \bI_{2n+1}^*$, we have
 $ \hat{u}_{\bI_{2n}}(x)\leq \hat{u}_{\bI_{2n+1}^*}(x)$, and specifically,  $ \hat{u}_{\bI_{2n}}(x_1)\leq \hat{u}_{\bI_{2n+1}^*}(x_1)$. Thus,
 \begin{align*}
 \hat{u}_{\R}(x_1)=\lim_{n\to \infty} \hat{u}_{\bI_{2n}}(x_1)\leq \lim_{n\to \infty} \hat{u}_{\bI_{2n+1}^*}(x_1)=\lim_{n\to \infty} \hat{u}_{\bI_{2n+1}}(x_2)=\hat{u}_{\R}(x_2).
 \end{align*}
This implies that $\hat u_{\R}(x)$ is non-decreasing in $x$. We can similarly show that $\hat u_{\R}(x)$ is non-increasing in $x$. Hence it is a constant function, as claimed.

  We now show that $\hat{u}_{\R}(x)\equiv u_*$. By the regularity estimate \cite{RS},   there is a constant  $C>0$ independent of $n\geq 1$ such that  for the bounded interval $\bI_n=(-n,n)$,  we have
  $$\|\hat{u}_{\bI_{2n}}\|_{C^{2s+\delta}(\bI_{n})}\leq C.$$
 It then follows easily that $\displaystyle \hat{u}_{\R}(x)=\lim_{n\to \infty} \hat{u}_{\bI_n}(x)$ is a classical solution of
 \begin{equation}\label{eeq 4.1-whole}
 \left\{ \arraycolsep=1pt
\begin{array}{lll}
\displaystyle   d (-\Delta)^s v= f(v ) \quad \  &{\rm in}\ \    \R,\\[2mm]
 \phantom{ (-\Delta)^s\    }
 v\gneqq  0\quad \ &{{\rm in}}\  \   \R.
\end{array}
\right.
\end{equation}
Since $ \hat{u}_{\R}(x)$ is a positive constant, we must have $\hat{u}_{\R}\equiv u_*$,
and so, due to the monotonicity of $\bI\mapsto \hat{u}_{\bI}$, we obtain
  $\hat{u}_{\bI}\to u_*$  locally uniformly in $\R$ as $\bI\to \R$.  \end{proof}

\subsection{Spreading-Vanishing dichotomy}

 Let $\bI$ be a finite open interval in $\R$. For $l>0$, set  $l \bI=\{l x:x\in\bI\}$
and $\phi_{l,i}(x)=l^{-\frac12} \phi_{s,\bI,i} (l^{-1}x)$ in $\R$, $i\in\mathbb N$, where $\{(\lambda_{s, i}(\bI), \phi_{s,\bI, i})\}$ is the sequence of eigenvalues and corresponding eigenfunctions of $(-\Delta)^s$ over $\bI$ with Dirichlet boundary conditions. Then  for $x=ly\in l\bI$,
 \begin{align*}
 (-\Delta)^s \phi_{l,i}(x)=l^{-\frac12} (-\Delta_x)^s \big(\phi_{s,\bI,i}(l^{-1}x)\big)&=l^{-2s-\frac12}  \big((-\Delta_y)^s (\phi_{s,\bI,i}\big) (y)
 \\&\, =\,  l^{-2s-\frac12} \lambda_{s,i}(\bI) \phi_{s,\bI,i}  (y)\, =l^{-2s} \lambda_{s,i}(\bI) \phi_{l, i}(x) ,
 \end{align*}
which implies that for any $i\in\N$,
\begin{equation}\label{eq 2.1-eig}
\lambda_{s,i}(l\bI)=l^{-2s} \lambda_{s,i}(\bI).
\end{equation}
In particular, if we take $\hat\bI=(0,1)$, then $\lambda_{s, i}(l\hat\bI)=l^{-2s}\lambda_{s,i}(\hat\bI)$ is strictly decreasing in $l\in (0,\infty)$ and
\[
\lim_{l\to\infty}\lambda_{s,i}(l\hat\bI)=0,\  \ \ \ \lim_{l\to 0}\lambda_{s,i}(l\hat\bI)=\infty \mbox{ for every } i\in\mathbb N.
\]
Moreover, it is easily seen that for any $\alpha\in\R$, with $I_\alpha=I+\alpha:=\{x:  x=y+\alpha,\ y\in \bI\}$, it holds
\[
\lambda_{s,  i}(\bI_\alpha)=\lambda_{s, i}(\bI) \mbox{ for every } i\in\mathbb N.
\]
It follows that with $\hat\lambda_1:=\lambda_{s,1}(\hat \bI)$,  for any finite interval $\bI$, we have
\[
\lambda_{s,1}(\bI)=|\bI|^{-2s}\hat\lambda_1.
\]
Therefore there exists a unique $l_*>0$ such that
\begin{equation}\label{l*}
\lambda_{s,1}(\bI)\begin{cases} >f'(0)/d &\mbox{ if \ } |I|<l_*,\\
=f'(0)/d &\mbox{ if \ } |I|=l_*,\\
<f'(0)/d &\mbox{ if \ } |I|>l_*.
\end{cases}
\end{equation}

\begin{theorem}\label{thm-s-v}
Assume that $s\in(0,1)$,  $f$ verifies ${\bf (f1)-(f4)}$, $h_0>0$   and    
$u_0  \in  \mathcal I_0 ((-h_0, h_0))$.  Let  $(g ,h ,u)$  be the unique solution of  the problem \eqref{eq 1.1}   and  $\cO_t:=(g(t), h(t))$, $\cO_\infty=(g_\infty,h_\infty)=:\displaystyle \lim_{t\to+\infty}\cO_t$.
Then the following dichotomy holds, namely,

\hspace{0.3cm} {\rm either (i)} \underline{\rm Vanishing:}  $h_\infty-g_\infty< + \infty$ and
	\[
	|\cO_\infty|\leq l_*, \mbox{\  \ $u(t,x)\rightarrow 0$ uniformly in $x\in\R$ as
	$t\rightarrow+\infty$},\]
	
	\hspace{0.3cm} {\rm or (ii)} \underline{\rm Spreading:}  $h_\infty-g_\infty=+\infty$ and
	\[
	\cO_\infty= \R,
	\mbox{\  \ $u(t,x)\rightarrow u_*$ locally uniformly in $\R$  as
	$t\rightarrow+\infty$.}
	\]

\end{theorem}
\noindent{\bf Proof. } (i) We prove the conclusion indirectly.  Suppose  that  $h_\infty-g_\infty< + \infty$ and $ |\cO_\infty|> l_*$, then
$$d \lambda_{s,1}(\cO_\infty)- f'(0)<0.$$
 Thus there exists $t_0>0$ such that $d \lambda_{s,1}(\cO_t)- f'(0)<0$ for $t\geq t_0$.  Let $\bu(t,x; \tilde u_0)$  be the solution of (\ref{eq 4.1}) with initial function $\tilde u_0$ and $\bI=\cO_{t_0}$.
Taking $\tilde u_0(x)=  u(t_0,x)$,  then it follows from  Lemma  \ref{lm comp-1}  that
$$u(t+t_0,\cdot)\geq  \bu(t,\cdot; u_0)\quad {\rm in}\ \ \R. $$
Since $d \lambda_{s,1}(\cO_{t_0})- f'(0))< 0$, by Proposition \ref{pr cyl-1} (ii) we have  $\bu(t,\cdot, u_0)\to v_{\cO_{t_0}}$ as $t\to+\infty$ ,
where $v_{\cO_{t_0}}$ is a positive solution of (\ref{eeq 4.1})  with $\bI=\cO_{t_0}$.
Then for large $t>0$, we derive  $\bu(t,\cdot, u_0)\geq  v_{\cO_{t_0}}/2$, and
 \begin{align*}
	h'(t)&=  \frac{ \mu}{2s}\int_{g(t)}^{h(t)}   u(t,x) |{h(t)}-x|^{-2s}  dx
	\\[1mm]& \geq  \frac{ \mu}{2s}\int_{g(t)}^{h(t)}   \frac{v_{\cO_{t_0}}}{2} |{h(t)}-x|^{-2s}  dx
	\\[1mm]&\geq  \frac{ \mu}{2s}\int_{g(t_0)}^{h(t_0)}   \frac{v_{\cO_{t_0}}}{2} |h_\infty-x|^{-2s}  dx
	=: C>0.
\end{align*}
Similarly we can show $-g'(t)\geq \tilde C>0$ for all large $t>0$.
 Therefore
$$  \lim_{t\to+\infty} h(t)=+\infty, \quad \lim_{t\to+\infty} g(t)=-\infty,$$
 a contradiction to the assumption that $h_\infty-g_\infty<\infty$. This proves   $|\cO_\infty|\leq l_*$.

 Let $\bu_1(t,x; u_0)$  be the solution of (\ref{eq 4.1}) with $\bI=\cO_{\infty}$.   It follows from Proposition \ref{pr cyl-1} (i)  that  $\bu_1(t,\cdot, u_0)\to 0$ as $t\to+\infty$ uniformly for $x\in\R$ since $d \lambda_{s,1}(\cO_\infty)- f'(0)\geq 0$.  By the maximum principle, we deduce $0\leq u(t,x)\leq  \bu_1(t,x; u_0)$ for $t\geq 0$ and $x\in\R$,  and so $u(t,x)\rightarrow 0$ uniformly in $x\in\R$ as
$t\rightarrow+\infty$.

\medskip

(ii) From  $|\cO_{t}|\to+\infty$ and the fact that $\lambda_{s,1}(\cO_t)\sim |\cO_t|^{-2s}$ as $t\to+\infty$, we see $\displaystyle \lim_{t\to\infty}d \lambda_{s,1}(\cO_t)=0$. Hence, there exists $t_0>0$ such that $d \lambda_{s,1}(\cO_t)- f'(0))<0$ for $t\geq t_0$.
As in the proof of part (i) above, we could show the following results by an indirect argument:
$$\lim_{t\to+\infty} h(t)=+\infty,\ \ \lim_{t\to+\infty} g(t)=-\infty. $$

Fix  $t_0>0$ large so that $d \lambda_{s,1}(\cO_{t_0})- f'(0))<0$, and let $\bu$ be the unique solution of
\begin{equation}\label{eq 4.1-n}
\left\{ \arraycolsep=1pt
\begin{array}{lll}
\displaystyle \partial_t u+ d (-\Delta)^s  u=f(u)\quad \  &{\rm in}\ \   (0,+\infty)\times \cO_{t_0},\\[2mm]
 \phantom{ (-\Delta)^s --\ \,   }
\displaystyle   u  = 0 \quad \ &{{\rm in}}\  \    ( 0,+\infty)\times \big(\R \setminus \cO_{t_0} \big), \\[2mm]
 \phantom{ (-\Delta)^s  \   }
  u(0,\cdot) = u(t_0,\cdot)  \quad \ &{{\rm in}}\  \ \cO_{t_0}.
\end{array}
\right.
\end{equation}
By Proposition \ref{pr cyl-1} (ii),  we have $\bu(t,x)\to v_{_{\cO_{t_0}}}$ as $t\to+\infty$.
The maximum principle implies that $u(t+t_0,x)\geq \bu(t,x)$ in $( 0,+\infty)\times  \R$, which implies
$$\liminf_{t\to\infty} u(t,x)\geq v_{_{\cO_{t_0}}}(x)\quad {\rm for}\ x\in\R. $$
In view of Proposition \ref{pr cyl-1} (iii), we can let $t_0\to+\infty$ to deduce $\displaystyle  \liminf_{t\to+\infty} u(t,x)\geq  u_*$ locally uniformly for $x\in\R$. On the other hand, it is clear that the solution of the ODE problem $w=f'(w)$ with $w(0)=1+\max u_0$ is a super solution of \eqref{eq 1.1}, which implies $\displaystyle \limsup_{t\to+\infty} u(t,x)\leq \lim_{t\to+\infty} w(t)=u_*$. Therefore, $\displaystyle \lim_{t\to+\infty} u(t,x)= u_*$ locally uniformly for $x\in\R$.
\hfill$\Box$\medskip

%\noindent{\bf Proof of Theorem \ref{teo 2}.} The proof follows by Proposition \ref{thm-vanishing} Directly. \hfill$\Box$

 \subsection{Spreading-vanishing criteria}

 \begin{theorem}\label{thm-criteria}
 Assume that   $s\in(0,1)$,  $f$ verifies ${\bf (f1)-(f4)}$,  $h_0>0$ and         $u_0\in \mathcal I_0((-h_0, h_0))$.
Let  $(u,g,h)$ be the unique solution of   problem \eqref{eq 1.1} and $l_*$ be given in \eqref{l*}. Then spreading always happens to $(u,g,h)$ regardless of the choice of
$u_0\in \mathcal I_0((-h_0, h_0))$ if $h_0\geq l_*/2$. If $h_0<l_*/2$,
 then there exists $\mu^*\in(0,+\infty)$ dependent on $u_0$ such that
 \[\begin{cases}
 \mu\leq \mu^*\implies \mbox{vanishing}, \\  \mu>\mu^*\implies \mbox{spreading}.
 \end{cases}
\]
 \end{theorem}

 The following lemma will be used to prove Theorem \ref{thm-criteria}.
 \begin{lemma}\label{lm-decay}
Assume that   $s\in(0,1)$,  $f$ verifies ${\bf (f1)-(f4)}$, $h_0>0$, $u_0\in\mathcal I_0((-h_0, h_0))$ and  $(g,h)\in \mathbb G_{h_0, \infty}\times \mathbb H_{h_0, \infty}$.
Denote $\cO_t=(g(t), h(t))$, $\displaystyle  \cO_{\infty}=\lim_{t\to+\infty}(g(t),h(t))$. Suppose that
$$|\cO_\infty|<l_* \ \mbox{ and } \ f'(0)<d\lambda_{s,1}(\cO_{\infty}),$$
and $u_{g,h}$ is the unique solution of
\begin{equation}\label{eq 4.1-e0}
 \left\{
\begin{array}{lll}
\partial_t  u  +d(-\Delta)^s  u=f(u)    \quad &{\rm for}\ \,   t\in(0,+\infty),\ x\in \cO_{t},\\[3mm]
 u(t,x)=0  \quad  &{\rm for}\ \, t\in(0,+\infty),\ \, x\in \R \setminus \cO_{t},\\[3mm]
 u(0,x)= u_0(x)  &{\rm for}\ \,  x\in\R.
\end{array}\right.
\end{equation}
 Then there exists $\tilde C_0>0$, $\epsilon>0$ and $T>0$  such that
\begin{equation}\label{e 4.10}
 0\leq  u_{g,h}(t,x) \leq \tilde C_0\,  e^{-\epsilon t} \rho (t,x)^{s}\ \mbox{ for }  x\in \cO_{t} \mbox{ and all } t>T,
\end{equation}
where   $\rho(t,x)=\min\{h(t)-x,\, x-g(t)\}$.
 \end{lemma}
 \begin{proof}
 Let $v_\infty(t,x)$ be the solution to
 \begin{equation}
 \left\{
\begin{array}{lll}
\partial_t  u  +d(-\Delta)^s  u=f(u)    \quad &{\rm for}\ \,   t\in(0,\infty),\ x\in \cO_{\infty},\\[3mm]
 u(t,x)=0  \quad  &{\rm for}\ \, t\in(0, \infty),\ \, x\in \R \setminus \cO_{\infty},\\[3mm]
 u(0,x)=u_0(x)  &{\rm for}\ \,  x\in\R.
\end{array}\right.
\end{equation}
Since   $f'(0)<d\lambda_{s,1}(\cO_{\infty})$, by Proposition \ref{pr cyl-1},
we have
\[
v_\infty(t,x)\to 0 \mbox{ uniformly for } x\in\R \mbox{ as } t\to+\infty.
\]
By the maximum principle we have
\[
u_{g,h}(t,x)\leq v_\infty(t,x) \ \mbox{ for } t>0,\ x\in\R.
\]
Denote $m(t):=\|v_\infty(t,\cdot)\|_\infty$. Then $m(t)\to 0$ as $t\to+\infty$.

Checking the proof of Lemma \ref{lem-I}, we see that there exists a constant $\hat C_0$ independent of $\tau>0$ such that
\begin{align*}
u_{g,h}(\tau+t,x)&\leq \hat C_0\big[\|u_{g,h}(\tau,\cdot)\|_\infty+\|f(u_{g,h}(\tau,\cdot))\|_\infty\big]\rho(\tau+t, x)^s\\
&\leq \hat C_0 M(\tau) \rho(\tau+t, x)^s \mbox{ for } t\in [1,2],\ \tau>0,\ x\in [g(\tau+t), h(\tau+t)],
\end{align*}
where
\[
M(\tau):=m(\tau)+\max_{\xi\in [0, m(\tau)]}|f(\xi)|\to 0 \mbox{ as } \tau\to+\infty.
\]
The above estimate implies
\begin{equation}\label{Mt}
u_{g,h}(t, x)\leq \hat C_0 M(t-1) \rho(t, x)^s \ \mbox{ for } t>2,\ x\in [g(t), h(t)].
\end{equation}
For the same reason, we have
\[
v_\infty(t,x)\leq \hat C_0 M(t-1) \rho(x)^s \ \mbox{ for } t>1,\ x\in [g_\infty, h_\infty],
\]
where $\rho(x):=\min\{x-g_\infty, h_\infty-x\}$.

Since
\[
\lim_{u\to 0} f(u)/u=f'(0) <d\lambda_{s,1}(\cO_{\infty}),
\]
we can choose $\epsilon>0$ small so that
\[
f(u)/u\leq f'(0)+\epsilon <d\lambda_{s,1}(\cO_{\infty})-\epsilon \mbox{ for } u\in (0, \epsilon].
\]
Let $\phi_{s,1}\geq 0$ be the principle eigenfunction associated to $\lambda_{s,1}(\cO_\infty)$ with $\|\phi_{s,1}\|_\infty=1$.
Then
\[
\bar u(t,x):=\epsilon e^{-\epsilon t}\phi_{s,1}(x)
\]
 satisfies
\[
\bar u_t+d(-\Delta)^s\bar u=[d\lambda_{s,1}(\cO_\infty)-\epsilon]\bar u\geq f(\bar u) \
 \mbox{ for } t>0,\ x\in\cO_\infty.
\]
Since $\phi_{s,1}(x)\geq \sigma_0 \rho(x)^s$ in $\cO_\infty$ for some $\sigma_0>0$, we have
\[
\bar u(0, x)\geq \epsilon\sigma_0\rho(x)^s \
 \mbox{ for } x\in\cO_\infty.
\]
We now fix $t_0\gg 1$ so that
\[
\hat C_0 M(t_0-1)<\epsilon\sigma_0.
\]
Then
\[
v_\infty(t_0,x)\leq \epsilon\sigma_0\rho(x)^s\leq \bar u(0,x).
\]
Therefore we can use the maximum principle to deduce
\[
v_\infty(t_0+t,x)\leq \bar u(t,x)=\epsilon e^{-\epsilon t}\phi_{s,1}(x) \ \mbox{ for } t>0,\ x\in \cO_\infty.
\]
It follows that
\[
m(t)=\|v_\infty(t,\cdot)\|_\infty\leq \epsilon e^{\epsilon t_0}e^{-\epsilon t} \ \mbox{ for } t>t_0.
\]
From the definition of $M(\tau)$ and ${\bf (f1)}$ we see that
\[
M(\tau)\leq K m(\tau) \ \mbox{ for some $K>0$ and all $\tau\geq t_0$.}
\]
Hence
\[
\hat C_0M(t-1)\leq \hat C_0 K m(t-1)\leq \tilde C_0 e^{-\epsilon t} \ \mbox{ for some $\tilde C_0>0$ and  all $t\geq t_0+1$}.
\]
It now follows from \eqref{Mt} that
\[
u_{g,h}(t, x)\leq \tilde C_0 e^{-\epsilon t} \rho(t, x)^s \ \mbox{ for } t>t_0+1,\ x\in [g(t), h(t)].
\]
 The proof is complete.
 \end{proof}
    \medskip

 We are now ready to prove Theorem \ref{thm-criteria}.

 \medskip

 \noindent
 {\bf Proof of Theorem \ref{thm-criteria}:} If $h_0\geq l_*/2$, then due to $h'(t)>0>g'(t)$ for $t>0$, we necessarily have $\infty\geq h_\infty-g_\infty>h(0)-g(0)=2h_0\geq l^*$. Hence we can use Theorem \ref{thm-s-v} to conclude that spreading always happens.\medskip

Next we consider the case $h_0<l_*/2$, and assume that this inequality is always satisfied in the remainder of this proof. We are going to complete
   the proof in four steps.\smallskip

 {\bf Step 1:} We show that vanishing happens for small $\mu>0$.

 Fix $\beta\in (0, \frac{l_*}2-h_0)$ and let
  $${\bar h}(t):=h_0+\beta (1-\frac{1}{2+t}),\quad {\bar g}(t):=-{\bar h}(t). $$
  Clearly $({\bar g,\bar h})\in \bG_{h_0}\times \bH_{h_0}$ and ${\bar h}_\infty-{\bar g}_\infty=2h_0+2\beta<l_*$. By Lemma \ref{lm-decay},
the solution $\bar u(t,x)$ of
\begin{equation}\label{eq 4.2-0}
 \left\{
\begin{array}{lll}
\partial_t  u  +d(-\Delta)^s  u=f(u)   \quad &{\rm for}\ \,   t\in(0,+\infty),\ \, x\in({\bar g}(t), {\bar h}(t)),\\[3mm]
 u(t,x)=0  \quad  &{\rm for}\ \, t\in(0,+\infty),\ \, x\in \R \setminus ({\bar g}(t), {\bar h}(t)),\\[3mm]
 u(0,x)=u_0(x)  &{\rm for}\ \,  x\in\R
\end{array}\right.
\end{equation}
satisfies, for some positive constants $\tilde C_0$, $ \epsilon$ and $T$,
\begin{equation}\label{e 4.10-1}
 \bar u(t,x)\leq   \tilde C_0 e^{-\epsilon t}  \bar\rho (t,x)^{s} \mbox{ for } x\in ({\bar g}(t), {\bar h}(t)) \mbox{ and } t>0,
\end{equation}
where $\bar \rho(t,x):=\min\{x-\bar g(t), \bar h(t)-x\}$.
Then
 \begin{align*}
 &\mu \int_{{\bar g}(t)}^{{\bar h}(t)}  \bar u(t,x) |{\bar h}(t)-x|^{-2s}  dx\leq \mu \tilde C_0 e^{-\epsilon t}   \int_{{\bar g}(t)}^{{\bar h}(t)}  \bar \rho (t,x)^s({\bar h} (t)-x)^{-2s} dx
  \\[1mm]&\leq  \mu \tilde C_0 e^{-\epsilon t}   \int_{{\bar g}(t)}^{{\bar h}(t)}  ({\bar h} (t)-x)^{-s} dx
  =  \frac{\mu \tilde C_0}{1-s} e^{-\epsilon t}   ({\bar h}(t)-{\bar g}(t))^{1-s}
  \\[1mm]&\leq  \mu L_0  e^{-\epsilon t}  \mbox{ for $t>T$ and $L_0:= \frac{\tilde C_0}{1-s} (2h_0+2\beta)^{1-s}$}.
 \end{align*}
 Since
 \[
 \underline \mu:=\frac\beta {L_0}\inf_{t>T}\frac{e^{\epsilon t}}{(2+t)^2}>0,
 \]
 we obtain, for $\mu\in (0, \underline\mu]$,
 \[
 \mu \int_{{\bar g}(t)}^{{\bar h}(t)}  \bar u(t,x) |{\bar h}(t)-x|^{-2s}  dx\leq \underline\mu L_0  e^{-\epsilon t}\leq \frac\beta{(2+t)^2}={\bar h}'(t) \ \mbox{ for } t>T.
 \]
 Similarly,
 \[
 -\mu \int_{{\bar g}(t)}^{{\bar h}(t)}  \bar u(t,x) |x-{\bar g}(t)|^{-2s}  dx\geq {\bar g}'(t) \ \mbox{ for } t>T.
 \]

 With $T$ given above, and $\rho(t,x):=\min\{x-g(t), h(t)-x\}$, the solution $(u,g,h)$ of \eqref{eq 1.1} satisfies,
 by \eqref{ab-2},
 \[
 u(t,x)\leq \frac{C_0}{\sqrt{t}}\,\rho(t,x)^s \ \mbox{ for } t\in (0, T],\ x\in (g(t), h(t)),
 \]
 where  $C_0>0$ is independent of $\mu>0$. It follows that
 \begin{align*}
 h(t)&\leq h_0+\mu \int_0^t \int_{g(\tau)}^{h(\tau)} \frac{C_0}{\sqrt{\tau}}[h(\tau)-x]^{-s} dxd\tau\\
&=h_0+\mu \int_0^t \frac{C_0}{\sqrt{\tau}}\frac{[h(\tau)-g(\tau)]^{1-s}}{1-s} d\tau\\
&\leq h_0+\mu \frac{[h(t)-g(t)]^{1-s}}{1-s}\int_0^t  \frac{C_0}{\sqrt{\tau}}d\tau  \\
&=h_0+\mu\hat C_0 \sqrt{t}\, [h(t)-g(t)]^{1-s} \ \mbox{ for } t\in (0, T],
\end{align*}
 Similarly,
\[
g(t)\geq -h_0-\mu\hat C_0 \sqrt{t}\, [h(t)-g(t)]^{1-s} \ \mbox{ for } t\in (0, T].
\]
Therefore
\[
 h(T)- g(T)\leq 2h_0+2\mu \hat C_0 \sqrt{T}\, [h(T)-g(T)]^{1-s} .
\]
Since $1-s\in (0, 1)$, the above inequality indicates that $h(T)-g(T)$ has an upper bound independent of $\mu\in (0, 1]$, say $h(T)-g(T)\leq L$
for $\mu\in (0, 1]$. Then
\[
h(T)-g(T)\leq 2h_0+2\mu C_0\sqrt{T} L \ \mbox{ for } \mu\in (0, 1].
\]
Therefore we can find $\underline \mu^*\in (0, \underline \mu]$ sufficiently small so that
\[
h(T)-g(T)\leq 2h_0+\beta/2 \ \mbox{ for } \mu\in (0, \underline\mu^*].
\]
It follows that $h(t)\leq h(T)\leq h_0+\beta/2=\bar h(0)<\bar h(t) $ for $t\in (0, T]$.
Similarly,
$g(t)>\bar g(t)$ for $t\in (0, T]$. We may now apply the maximum principle to conclude that, whenever $\mu\in (0, \underline\mu^*]$,
\[
[g(t), h(t)]\subset (\bar g(t), \bar h(t)),\ u(t,x)\leq\bar u(t,x) \ \mbox{ for } t\in (0, T], \ x\in (g(t), h(t)).
\]
In view of our earlier estimates for $\bar h'(t)$ and $\bar g'(t)$ with $t>T$, the above estimates allow us to regard $t=T$ as the initial time and
 apply the comparison principle (see Remark \ref{rm-cp}),   to conclude that, when $\mu\in (0,\underline\mu^*]$,
  \begin{align*}
& {\bar g}(t)  	< g(t)<h(t)<{\bar h}(t) \ \mbox{ for } t>T,\\[1mm]
&0\leq u(t,x)\leq \bar u(t,x) \ \mbox{ for } t> T,\ x\in\R.
  \end{align*}
Hence $u(t,x)\rightarrow 0$ uniformly in $x\in\R$ as
$t\rightarrow+\infty$, and  vanishing happens.

  \medskip

 {\bf Step 2:} We show that for all large $\mu>0$, spreading happens.

With $\cO_0=(-h_0,h_0)$,   the problem
 $$ \left\{
\begin{array}{lll}
\partial_t  u  +d(-\Delta)^s  u=f(u)    \quad &{\rm for}\ \,   (t,x)\in(0,2] \times \cO_0,\\[3mm]
 u(t,x)=0  \quad  &{\rm for}\ \, t\in(0,1],\ \, x\in \R \setminus \cO_0,\\[3mm]
 u(0,x)=\tilde u_0(x)  &{\rm for}\ \,  x\in \R
\end{array}\right.
$$
has a unique solution  $\bu(x,t;\cO_0)$ and there exists $c_1>0$ such that
$$\bu(t,x;\cO_0)\geq c_1\rho_{_{\cO_0}}(x)^s\qquad {\rm for}\ \,   (t,x)\in [1,2] \times \cO_0,$$
where $\rho_{_{\cO_0}}(x)=\min\{h_0-x,\, x-h_0\}$.
By the maximum principle Lemma \ref{lm comp-1}, we have $u(t,x)\geq \bu(t,x;\cO_0)$ for $t\in (0, 2]$ and $x\in \cO_0$. It follows that
 \begin{align}\label{es 4.3}
 u(t,x)\geq \bu(t,x;\cO_0)\geq  c_1\rho_{_{\cO_0}}(x)^s \quad {\rm for}\ \,   (t,x)\in [1,2] \times \cO_0.
 \end{align}

Direct calculation gives, for $t\in[1,2]$,
 \begin{align*}
h'(t)&
=  \mu \int_{g(t)}^{h(t)}   u(t,x) |{h(t)}-x|^{-2s}  dx
   \\[1mm]&\geq  c_1      {\mu }\int_{g(t)}^{h(t)} \rho_{_{\cO_0}}(x)^s    |{h(t)}-x|^{-2s}  dx\\[1mm]
  &\geq   \frac{c_1    \mu}{[h(2)-g(2)]^{2s}}\int_{-h_0}^{h_0} \rho_{_{\cO_0}}(x)^s     dx,
 \end{align*}
 where $c_1>0$ is independent of $\mu$.

 By Theorem \ref{thm-s-v}, to show spreading happens to $(u,g,h)$, it suffices to show $h(2)-g(2)>l^*$. We show this is the case if $\mu>0$ is large enough.
 Otherwise, we have $h(2)-g(2)\leq l^*$ for all $\mu>0$, and so the above estimate yields
 \[
 h'(t)\geq \frac{c_1    \mu}{(l^*)^{2s}}\int_{-h_0}^{h_0} \rho_{_{\cO_0}}(x)^s     dx=c_2\mu \ \mbox{ for } t\in [1,2],
 \]
 with $c_2>0$ independent of $\mu$. It follows that
 \[
 h(2)\geq h(1)+c_2\mu \to+\infty \mbox{ as } \mu\to+\infty,
 \]
 which clearly contradicts to $h(2)\leq h(2)-g(2)\leq l^*$.
  Therefore,  spreading happens for all large $\mu$.

 \smallskip

{\bf Step 3: } We show that if spreading happens to $(u,g,h)$ with $\mu=\mu_0$, then spreading happens for any $\mu>\mu_0$.

 For notational convenience, we denote the solution of \eqref{eq 1.1} with $\mu=\mu_0$ by $(u^0, g^0, h^0)$. Since spreading happens, there exists $t_0>0$ such that
 \[
 h^0(t_0)-g^0(t_0)>l^*.
 \]
 Fix $\mu>\mu_0$ and let $(u,g,h)$ be the solution of \eqref{eq 1.1} with such a $\mu$. By Theorem \ref{thm-s-v}, it suffices to show that $h(t_0)-g(t_0)\geq h^0(t_0)-g^0(t_0)$.
 This is actually a consequence of the comparison principle in Remark \ref{rm-cp}, although we cannot apply it directly due to $h(0)=h^0(0)=h_0$ and $g(0)=g^0(0)=-h_0$.
 Let $(u_n, g_n, h_n)$ be the solution of \eqref{eq 1.1} with $h_0$ replaced by $h_0+1/n$. Then the above mentioned comparison principle can be applied to deduce
 \[
 (g^0(t), h^0(t))\subset (g_n(t), h_n(t))\subset (g_{n-1}(t), h_{n-1}(t)),\ \ u^0(t,x)\leq u_n(t,x)\leq u_{n-1}(t,x) \]
  for $ t>0,\ x\in\R,\ n\geq 2$.

 Using the interior and boundary estimates for $u_n$, much as in the proof of Theorem \ref{thm-semi}, we can show that $(u_n, g_n, h_n)$ converges to
 some $(\tilde u, \tilde g, \tilde h)$ uniformly for  $x\in \R$ and $t$ in any bounded subset of $[0,\infty)$ as $n\to\infty$, and $(\tilde u, \tilde g, \tilde h)$ solves \eqref{eq 1.1}. By uniqueness, we necessarily have $(\tilde u, \tilde g, \tilde h)=(u,g,h)$.
 Hence
 \[h(t_0)-g(t_0)=\lim_{n\to+\infty} [h_n(t_0)-g_n(t_0)]\geq h^0(t_0)-g^0(t_0)>l^*,
 \]
 as desired.
 \medskip

 {\bf Step 4:} We prove the existence of $\mu^*>0$ with the desired properties.

 Define
 \[
 \Sigma:=\{\mu>0: \mbox{ Spreading happens to \eqref{eq 1.1}}\}.
 \]
 By Steps 1, 2 and 3 above, we know that $\Sigma$ is an unbounded interval of the form $(\mu^*,\infty)$ or $[\mu^*,\infty)$, with $0< \mu^*<\infty$. We show next that $\mu^*\not\in\Sigma$, which would complete the proof of the theorem.

  We argue indirectly by assuming that spreading happens for $\mu=\mu_*$. Denote the solution of \eqref{eq 1.1} with $\mu=\mu^*$ by $(u^*, g^*, h^*)$. Then there exists $t_1>0$ such that
  \[
  h^*(t_1)-g^*(t_1)>l^*.
  \]

  For any $\mu\in (0, \mu^*)$ we can use the same approximation argument as in Step 3 and the comparison principle to show that the corresponding solution $(u,g,h)$ with this $\mu$ satisfies
  \[
  (g(t), h(t))\subset (g^*(t), h^*(t)),\ \ u(t,x)\leq u^*(t,x) \ \mbox{ for } t>0, \ x\in\R.
  \]
  Moreover, if $\{\mu_n\}$ is an increasing sequence converging to $\mu^*$ as $n\to\infty$, and $(u_n, g_n, h_n)$ is the
  corresponding solution of \eqref{eq 1.1} with $\mu=\mu_n$, then for the same reason,
  \[
  (g_n(t), h_n(t))\subset (g_{n+1}(t), h_{n+1}(t))\subset (g^*(t), h^*(t)),\ \ u_n(t,x)\leq u_{n+1}(t,x) \leq u^*(t,x)
  \]
  for $ t>0, \ x\in\R,\ n\geq 1$. We may now use the interior and boundary estimates for $u_n$ as in Step 3 above to conclude that $(u_n, g_n, h_n)$ converges to
  some $(\tilde u, \tilde g, \tilde h)$ uniformly for $t\in [0, t_1]$ and $x\in \R$ as $n\to\infty$, and $(\tilde u, \tilde g, \tilde h)$ solves \eqref{eq 1.1} with $\mu=\mu^*$. By uniqueness, we necessarily have $(\tilde u, \tilde g, \tilde h)=(u^*,g^*,h^*)$. Thus
  \[
  h_n(t_1)-g_n(t_1)\to h^*(t_1)-g^*(t_1)>l^* \ \mbox{ as } n\to\infty.
  \]
  It follows that for all large $n$, $h_n(t_1)-g_n(t_1)>l^*$, which implies, by Theorem \ref{thm-s-v}, that spreading happens to $(u_n, g_n, h_n)$ for such $n$, and so $\mu_n\in\Sigma$. On the other hand, from the definition of $\mu^*$, the assumption $\mu_n<\mu^*$ implies $\mu_n\not\in\Sigma$. This contradiction completes our proof.
 \hfill$\Box$\medskip

Clearly Theorems \ref{teo 2.2} and \ref{teo 2} follow directly from the results in this section.

\section{Declarations and Statements}

\noindent {\bf  Conflicts of interest:} The authors declare that they have no conflicts of interest regarding this work.

\medskip

\noindent {\bf Data availability:} This paper has no associated data.

\medskip

\noindent{\bf Funding:}  
This work is supported  
by the National Natural Science Foundation of China (No. 12361043) and the Australian Research Council.

\end{document}